\documentclass[12pt,a4paper]{amsart}

\usepackage{amssymb}
\usepackage{amscd}
\usepackage{hyperref}

\def\frak{\mathfrak}
\def\Bbb{\mathbb}
\def\Cal{\mathcal}

\let\phi\varphi

\newcommand{\x}{\times}
\renewcommand{\o}{\circ}

\newcommand{\al}{\alpha}
\newcommand{\be}{\beta}

\newcommand{\ep}{\epsilon}

\newcommand{\la}{\lambda}

\newcommand{\ph}{\phi}
\newcommand{\ps}{\psi}

\newcommand{\si}{\sigma}
\newcommand{\ze}{\zeta}
\newcommand{\Ga}{\Gamma}
\newcommand{\La}{\Lambda}

\newcommand{\Om}{\Omega}

\newcommand{\Up}{\Upsilon}
\def\Rho{\mbox{\textsf{P}}}

\newcommand{\barm}{\overline{M}}
\newcommand{\vol}{\operatorname{vol}}
\newcommand{\im}{\operatorname{im}}
\newcommand{\Ric}{\operatorname{Ric}}

\numberwithin{equation}{section}

\newcounter{theorem}
\numberwithin{theorem}{section}
\newtheorem{thm}[theorem]{Theorem}
\newtheorem*{thm*}{Theorem \thesubsection}
\newtheorem{lemma}[theorem]{Lemma}
\newtheorem{prop}[theorem]{Proposition}

\newtheorem*{lemma*}{Lemma \thesubsection}
\newtheorem*{prop*}{Proposition \thesubsection}
\newtheorem*{cor*}{Corollary \thesubsection}

\theoremstyle{definition}
\newtheorem{definition}[theorem]{Definition}
\newtheorem*{definition*}{Definition \thesubsection}

\newtheorem*{example*}{Example \thesubsection}
\theoremstyle{remark}
\newtheorem{remark}[theorem]{Remark}
\newtheorem*{remark*}{Remark \thesubsection}

\def\sideremark#1{\ifvmode\leavevmode\fi\vadjust{\vbox to0pt{\vss% the remark
 \hbox to 0pt{\hskip\hsize\hskip1em%                          will appear only
 \vbox{\hsize3cm\tiny\raggedright\pretolerance10000%          on the side
  \noindent #1\hfill}\hss}\vbox to8pt{\vfil}\vss}}}%
\begin{document}
\renewcommand{\today}{}
\title{Projective Compactifications\\ and Einstein metrics}

\author{Andreas \v Cap and A.\ Rod Gover}

\address{A.\v C.: Faculty of Mathematics\\
University of Vienna\\
Oskar--Morgenstern--Platz 1\\
1090 Wien\\
Austria\\
A.R.G.:Department of Mathematics\\
  The University of Auckland\\
  Private Bag 92019\\
  Auckland 1142\\
  New Zealand;\\
Mathematical Sciences Institute\\
Australian National University \\ ACT 0200, Australia} 
\email{Andreas.Cap@univie.ac.at}
\email{r.gover@auckland.ac.nz}

\begin{abstract}
  For complete affine manifolds we introduce a definition of
  compactification based on the projective differential geometry
  (i.e.\ geodesic path data) of the given connection. The definition
  of projective compactness involves a real parameter $\alpha$ called
  the order of projective compactness. For volume preserving
  connections, this order is captured by a notion of volume
  asymptotics that we define. These ideas apply to complete
  pseudo-Riemannian spaces, via the Levi-Civita connection, and thus
  provide a notion of compactification alternative to conformal
  compactification. For many orders $\alpha$, we provide an asymptotic
  form of a metric which is sufficient for projective compactness of
  the given order, thus also providing many local examples.

  Distinguished classes of projectively compactified geometries of
  orders one and two are associated with Ricci-flat connections and
  non--Ricci--flat Einstein metrics, respectively. Conversely, these
  geometric conditions are shown to force the indicated order of
  projective compactness. These special compactifications are shown to
  correspond to normal solutions of classes of natural linear PDE
  (so-called first BGG equations), or equivalently holonomy reductions
  of projective Cartan/tractor connections. This enables the
  application of tools already available to reveal considerable
  information about the geometry of the boundary at infinity. Finally,
  we show that metrics admitting such special compactifications always
  have an asymptotic form as mentioned above.
\end{abstract}

\maketitle


\pagestyle{myheadings} \markboth{\v Cap, Gover}{Projective
  compactifications} 

\thanks{Both authors gratefully acknowledge support from the Royal
  Society of New Zealand via Marsden Grants 10-UOA-113 and 13-UOA-018;
  A\v C gratefully acknowledges support by project P23244-N13 of the
  ``Fonds zur F\"orderung der wissenschaftlichen For\-schung'' (FWF)
  and also the hospitality of the University of Auckland. }

\section{Introduction}\label{1}

Let $\barm$ be an $(n+1)$-dimensional manifold with boundary $\partial
M$. A defining function $r$ for the boundary is a smooth function on
$\barm$, with zero locus $\partial M$, and such that $dr$ is nowhere
vanishing on $\partial M$. Recall that a pseudo-Riemannian metric
$g^o$ on the interior $M$ of $\barm$ is said to be conformally compact
if $g:=r^2 g^o$ extends to a pseudo--Riemannian metric on $\barm$,
where $r$ is a defining function for the boundary; this extension
meaning that $g$ is smooth and non-degenerate up to the boundary. At
points where the boundary conormal is not null, the restriction of $g$
determines a conformal structure on $\partial M$ (that is independent
of the choice of $r$) and the conformal compactification provides a
geometric framework for relating conformal geometry, and associated
field theories, to the asymptotic phenomena of the interior
(pseudo-)\-Riemannian geometry of one higher dimension.  This notion
had its origins in the work of Newman and Penrose, for treating four
dimensional spacetime physics, and has remained extremely important
for general relativity and related questions
\cite{Chrusciel,Fr,Friedrich,LeBrunH,Penrose125}. 
Conformally compact geometries have also proved
to be a powerful tool for conformal invariant theory
\cite{FGast,FGbook}, the geometric scattering program and
related analysis \cite{GrZ,GuillamouAB,Ma-hodge,Melrose,Vasy}, and the
AdS/CFT program from physics
\cite{AdSCFTreview,deHaro,Henningson,Witten}.

Considering other geometric compactifications of complete metrics
should be very useful for many of these directions. Here we develop an
effective version of this idea based on projective geometry, thereby
also extending the concept to manifolds endowed with a complete affine
connection. Recall that torsion free connections are projectively
equivalent if they share the same geodesics up to reparametrisation;
the resulting emphasis on the role of geodesics seems particularly
natural for general relativity and related geometric analysis. The
resulting compactifications exhibit substantial differences from
conformal compactifications. This can for example be seen from the
description of the natural projective compactification of Minkowski
space in Section \ref{3.5}, see in particular Remark
\ref{rem3.5b}. More information on relations and differences between
projective and conformal compactifications can be found in Remark
\ref{rem2.3}.

In special settings projective compactifications have arisen in the
literature.  Indeed in Chapter 4 of Fefferman and Graham's book
\cite{FGbook} a ``projectively compact metric'' is given by an
explicit formula linked to the formula for their ``ambient metric''.
Einstein and Ricci flat projectively compact structures are seen to
arise naturally via suitable holonomy reductions of the canonical
Cartan/tractor connection \cite{ageom,CGM}; each such reduction is
equivalent to a special, so-called normal, solution of a certain
overdetermined PDE, a point we shall take up below. These results
follow a similar story for Einstein conformally compact metrics
\cite{Go-al}, and are part of a fascinating very general picture
\cite{hol-red}.  Extending a classical theme, recently there has been
increased interest in the interaction between projective and
pseudo-Riemannian geometries including in the context of general
relativity \cite{BalMatveev,BDE,EastwoodNotes,HL,HL2,Mikes,NurMet}. We
also see the current work as adding to this developing picture.

Let us summarize the program to be followed and the results
obtained. An affine connection on $M$ is projectively compact if its
projective class extends to the boundary in a suitable way. This is
made precise in Definition \ref{def2.1}, and in this a
parameter $\alpha\in \mathbb{R}_+$ is involved. For connections
preserving a volume density the number $\alpha$ controls the volume
growth asymptotics: smaller values of $\alpha$ are associated with
larger volume growth near the boundary, see Proposition \ref{prop2.2}.
The notion of volume asymptotics is defined in a general
context in Definition \ref{def2.2}, since this enables rather general
comparison; for example between the volume growth of a projective
compactification and the usual conformal compactification.

In Section \ref{2.1a} we treat completeness. It shown that for
$\al\leq 2$, projectively compact connections are automatically
complete, and essentially the converse holds. This is the content of
Proposition \ref{prop2.1a}.

A Levi-Civita connection can be projectively compact and hence we come
to the notion of projective compactness of a pseudo-Riemannian metric
in Definition \ref{def2.3}. A main result of the article is Theorem
\ref{thm2.3} which describes an asymptotic form of a metric depending
on $\alpha\in (0,2]$ which, for many values of $\al$, is sufficient for
  projective compactness of order $\alpha$. This Theorem provides a
  class of projectively compact metrics which can be used and treated
  in a manner similar to conformally compact metrics. In particular,
  it provides a large number of local examples of such metrics.

There is an alternative interpretation of volume growth, based on the
concept of defining densities, which is crucial for the second part of
the article. The line bundle $\mathcal{E}(\alpha)$ of densities of
projective weight $\alpha$ is defined in Section \ref{2.2}. A
``defining density'' $\sigma$ for the boundary is a section of such a
density bundle having $\partial M$ as its zero locus, and with the derivative
of $\si$  nowhere zero along $\partial M$. If an affine connection
admitting a parallel volume density 
is projectively compact of order
$\alpha$ then one obtains a defining density $\si\in \Gamma(
\mathcal{E} (\alpha))$ (unique up to constant multiples) for the
boundary.

It is thus natural to single out particularly nice projective
compactifications by imposing projectively invariant differential
equations on the natural defining density.  In particular there are
canonical overdetermined linear equations (called first BGG equations)
available in the case that $\alpha=1$ and $\alpha=2$
\cite{ageom,CSS}. These equations are rather well studied, in
particular it is known that there is a subclass of so--called normal
solutions which correspond to reductions of projective holonomy,
i.e.~the holonomy of the canonical Cartan/tractor connection
associated to the projective structure. Using the available results we
obtain the following picture: 
\begin{itemize}
\item On densities of weight 1 the BGG equation is second order and
  the natural defining density $\si\in \Gamma (\mathcal{E}(1))$ of a
  projectively compact space $(M,\nabla)$ satisfies this equation, if
  and only if $\nabla$ is Ricci flat. Furthermore the boundary
   $\partial M$ is totally geodesic
  and inherits a natural projective structure. The solution $\si$ in
  this case is necessarily ``normal'' thus giving rise to a reduction
  of projective holonomy on $\barm$, see Theorem \ref{thm3.2a}.
  
\item On densities of weight 2 the BGG equation is third order and for
  $\tau\in \Gamma (\mathcal{E} (2))$ a nontrivial solution means that,
  for the corresponding connection on $M$, the symmetrized covariant
  derivative of the Ricci curvature vanishes. If the solution is
  normal and satisfies a suitable non-degeneracy condition, then we
  get a non--Ricci--flat Einstein metric on $M$, with Levi--Civita
  connection in the projective class and a conformal structure on the
  boundary, see Theorem \ref{thm3.3a}. In Proposition \ref{prop3.3},
  we show that in the latter case the metric has asymptotics of the
  form discussed in Theorem \ref{thm2.3} for $\alpha=2$.
\end{itemize}

 The relation between the geometric conditions of Ricci--flatness,
 respectively being non--Ricci--flat Einstein, and the order of
 projective compactness exhibited in the above results is of a deep
 and fundamental nature; it is not due to any choice being made
 concerning the equations on the defining functions. This is shown by
 the following strong converses to the above results. For these
 results, we assume that for $\barm=M\cup\partial M$ we have an affine
 connection $\nabla$ on $M$ with the following properties: it
 preserves a volume density, it does not extend to any neighborhood of
 a boundary point, but its projective class does extend to all of
 $\barm$. Under these assumptions, we prove in Theorems \ref{thm3.2b}
 and \ref{thm3.3b}:
\begin{itemize}
\item If $\nabla$ is Ricci flat then it is projectively compact of
  order $\al=1$. Moreover, the unique (up to constant factors)
  non--zero section of $\Cal E(1)$ which is parallel for $\nabla$
  extends to a defining density for $\partial M$ which solves the
  relevant first BGG equation. Thus one automatically is in the
  setting of Theorem \ref{thm3.2a}.

\item If $\nabla$ is the Levi--Civita connection of a non--Ricci--flat
  Einstein metric (or, equivalently, an affine connection with
  parallel, non--degenerate Ricci curvature), then it is projectively
  compact of order $\al=2$. Moreover, the unique (up to constant
  factors) non--zero section of $\Cal E(2)$ which is parallel for
  $\nabla$ extends to a defining density for $\partial M$ which solves
  the first BGG equation on $\Cal E(2)$. Thus one automatically is in
  the setting of Theorem \ref{thm3.3a}.
\end{itemize}

Our last main result concerns projectively compact Ricci flat metrics
of any signature $(p,q)$, the model example of which is discussed in
Section \ref{3.4}. From the earlier results already mentioned, we know
that the order of projective compactness has to be one, there is a
natural defining density in $\Ga(\Cal E(1))$ which solves the first
BGG equation, and that the boundary $\partial M$ inherits a projective
structure. As an additional ingredient we use the fact that the metric
determines a normal solution to the projective metricity equation
(which again is a first BGG equation). Using this, we prove:
\begin{itemize}

\item The natural projective structure on the boundary inherits a
  reduction of projective holonomy to the orthogonal group $SO(p,q)$
  (see Theorem \ref{thm3.5}). This means that the boundary is {\em
    projective almost Einstein} which by \cite{ageom,hol-red} implies
  a stratification of $\partial M$, called a ``curved orbit
  decomposition'', which is explicitly described in Theorem
  \ref{thm3.6}. The open curved orbits are Einstein (never Ricci-flat)
  of signature $(p-1,q)$, respectively $(p,q-1)$, depending on whether
  the boundary points are the limits of space--like or time--like
  geodesics. The closed curved orbits consist of limit points of
  interior null geodesics and inherit a conformal structure of
  signature $(p-1,q-1)$.

\item In Proposition \ref{prop3.7} we show that, locally around the points of
  $\partial M$ which lie in open curved orbits, the interior metric
  has asymptotics of the form discussed in Theorem \ref{thm2.3} with
  $\alpha=1$.
\end{itemize}

\section{Projective compactifications}\label{2}

\subsection{Projectively compact affine connections}\label{2.1} 
Throughout this article smooth will mean $C^\infty$, we will refer to
linear connections on the tangent bundle on a manifold as
\textit{affine connections} and all such connections will be assumed
to be torsion free. Recall that two such connections are called
\textit{projectively equivalent}, if and only if they have the same
geodesics up to paramerization. Equivalently, their contorsion tensor
can be expressed in terms of a one--form $\Up$ in the form
$$
\hat\nabla_{\xi}\eta=\nabla_\xi\eta+\Up(\xi)\eta+\Up(\eta)\xi,
$$ 
where $\xi$, and $\eta$ are tangent vector fields.
We will formally write this relation as $\hat\nabla=\nabla+\Up$ from
now on. 

In the setting of a manifold $\barm$ of dimension $n+1$ with boundary
$\partial M$ and with interior $M$, the basic question we study in
this article is the following. Suppose we have a connection $\nabla$
on $M$ which does not extend to $\barm$, for example because it is
complete. Can we projectively modify it to a connection $\hat\nabla$,
which extends to $\barm$? Motivated by the concept of conformal
compactifications of Riemannian metrics in a similar setting, we
formulate this in terms of local defining functions $\rho$ for the
boundary. So we assume that $\rho$ is a smooth, real valued function
defined on some open subset $U\subset\barm$, with non--negative values,
such that $\rho^{-1}(\{0\})=\partial M\cap U$, and such that $d\rho$ is
nowhere vanishing on $\rho^{-1}(\{0\})$. We will be mainly interested
in the following condition in the cases $\al=2$ and $\al=1$, the other
cases are included for completeness.

\begin{definition}\label{def2.1}
Let $\al$ be a positive real number.  An affine connection $\nabla$ on
$M$ is called \textit{projectively compact} of order $\al\in
\mathbb{R}_+$ if for any $x\in\partial M$, there is a neighborhood $U$
of $x$ in $\barm$ and a defining function $\rho:U\to\Bbb R$ for the
boundary such that the connection
\begin{equation}\label{pc-def}
\hat\nabla=\nabla+\tfrac{d\rho}{\al\rho}
\end{equation}
on $U\cap M$ extends to all of $U$, i.e.~for arbitrary vector fields
$\xi$ and $\eta$ which are smooth up to the boundary also
$\hat\nabla_\xi\eta$ is smooth up to the boundary. 
\end{definition}

Observe first that for $\hat\nabla$ extending to $U$, it suffices to
require smoothness of $\hat\nabla_\xi\eta$ up to the boundary for the
elements of a local frame which is smooth up to the boundary or
equivalently smoothness of the Christoffel symbols of $\hat\nabla$ in
some local chart up to the boundary. Moreover, in case $\hat\nabla$
extends, the extension is uniquely determined by $\nabla$ and $\rho$
since $U\cap M$ is dense in $U$.

Given one defining function $\rho$, any other defining function for
the boundary can be locally written as $\tilde\rho=e^f\rho$ for some
smooth function $f$ on $U$. One immediately computes that
$\tfrac{d\tilde\rho}{\al\tilde\rho}=\tfrac{d\rho}{\al\rho}+\tfrac{1}\al
df$. Thus the question whether $\hat\nabla$ extends to the boundary is
actually independent of the defining function $\rho$. Note however
that the parameter $\al$ cannot be eliminated in a straightforward
way, since this would amount to replacing $\rho$ by some power of
$\rho$, which then cannot be a defining function.

\subsection{Volume asymptotics}\label{2.2}
If a linear connection $\nabla$ on $M$ is projectively compact of any
order $\al$ as in Definition \ref{def2.1}, then the projective
structure on $M$ defined by $\nabla$ extends to $\barm$, and this
extension is evidently unique. We next want to show that, for
connections admitting a parallel volume density, apart from the
extension of the projective structure, projective compactness of order
$\al$ amounts to a fixed growth rate of the volume towards the
boundary.

Recall that on a manifold $N$ endowed with a projective structure,
there is a standard notion of \textit{projective densities}. For any
real number $w\in\Bbb R$, one has the bundle $\Cal E(w)$ of densities
of \textit{projective weight} $w$. These bundles are always trivial,
but there is no preferred trivialization. Since they can be defined as
associated bundles to the linear frame bundle, any connection $\nabla$
in the projective class induces linear connections on all density
bundles which will be denoted by the same symbol. If
$\hat\nabla=\nabla+\Up$ for $\Up\in\Om^1(M)$ in the sense introduced
in Section \ref{2.1}, then the induced connections on $\Cal E(w)$ are
related by
\begin{equation}\label{denstrans}
\hat\nabla_\xi \si=\nabla_\xi\si+w\Up(\xi)\si \quad
\text{for\ }\si\in\Ga(\Cal E(w)),\xi\in\frak X(M). 
\end{equation}
This easily implies (see e.g.~\cite{CGM}) that given an arbitrary
nowhere vanishing section of $\Cal E(w)$ with $w\neq 0$, there is a
unique connection in the projective class for which this section is
parallel. In this situation, we will call the nowhere vanishing
section a \textit{scale} and refer to the resulting connection as
\textit{the connection determined by that scale}.

In order to allow comparison to the case of conformal structures, for
which there also is an established convention, we put this into a more
general context.  Recall that on a general (possibly non--oriented)
smooth manifold $N$ of dimension $n+1$ there is a natural line bundle
whose sections can be canonically integrated. This can be defined as
an associated bundle to the linear frame bundle of $N$ and if $N$ is
orientable, a choice of orientation identifies this bundle with the
bundle $\La^{n+1}T^*N$ of $(n+1)$--forms, see \cite[Section
  10]{Michor}. We will call this the bundle $\vol(N)$ of
\textit{volume densities} (avoiding the common name ``$1$--densities''
which might lead to confusion with the conventions mentioned
above). As above, any linear connection $\nabla$ on $TN$ induces a
linear connection on the line bundle $\vol(N)$, which we will denote
by the same symbol. We shall call the connection $\nabla$
\textit{special} if there is a non--zero section of $\vol(N)$ which is
parallel for the induced connection. Such a section then is clearly
uniquely determined up to a constant factor. For this (or its roots as
appropriate) we may use the term ``the canonical density'' determined
by a special affine connection. This slight abuse of language should
cause no confusion.

From the construction of the bundle of volume densities it follows
that $\vol(N)$ is always a trivial bundle, but there is no canonical
trivialization. Thus one can form powers of this bundle with any real
number as an exponent; this is easily established via the language of
associated bundles. The relation to projective densities then can be
simply expressed as $\Cal E(w)=\vol(N)^{-w/(n+2)}$ if $\dim(N)=n+1$. 

Returning to our standard setting, this allows us to define the notion
of volume asymptotics for special linear connections on the interior.
\begin{definition}\label{def2.2} 
Let $\barm$ be a smooth manifold with boundary $\partial M$ and
interior $M$, and let $\nabla$ be a special affine connection on
$M$. Then $\nabla$ is said to have \textit{volume asymptotics of order
  $\be$} if and only if for each point $x_0\in\partial M$, there is an
open neighborhood $U$ of $x_0$ in $\barm$, a local defining function
$\rho$ for $\partial M$, and a nowhere vanishing section $\nu$ of
$\vol(\barm)$ over $U$ such that the section $\rho^{-\be}\nu|_{U\cap
  M}$ of $\vol(M)$ is parallel for $\nabla$.
\end{definition}

Note that the number $\beta$ is independent of the choice of defining
function $\rho$, and corresponding section $\nu$.  There is an
alternative interpretation of volume asymptotics which will be very
useful for our purposes. This is based on the notion of defining
densities. Given a section $\nu$ of $\vol(\barm)$, one can naturally
view $\nu^{r}$ as a section of $\vol(\barm)^{r}$ for any $r\in\Bbb
R$. In particular, a choice of non--vanishing section of $\vol(\barm)$
gives rise to non--vanishing sections of all powers and for a special
affine connection on $M$ there are parallel sections of all density
bundles. For any fixed bundle these are unique up to constant
multiples.

Now there is a well defined notion of a defining density for $\partial
M$. Saying that $\si\in\Ga(\vol(\barm)^r)$ is a defining density means
that the zero locus of $\si$ coincides with $\partial M$ and with
respect to some (or equivalently any) local trivialization of
$\vol(\barm)^r$ around points in $\partial M$, $\si$ is represented by
a defining function. Otherwise put, for any locally non--vanishing
section $\hat\si$ of $\vol(\barm)^{r}$, the uniquely defined function
$\rho$ such that $\si=\rho\hat\si$ must be a defining function for
$\partial M$. The crucial point to notice here is that this pins down
$r$. If $\si$ is a defining density, then no power $\si^t$ for $t\neq
1$ can be a defining density for its zero locus.

We are now ready to clarify the relation between the order of
projective compactness and volume asymptotics.

\begin{prop}\label{prop2.2}
Let $\barm$ be a smooth $n+1$--dimensional manifold with boundary
$\partial M$ and with interior $M$.

(i) If $\nabla$ is a special linear connection on $M$, which is
projectively compact of order $\al>0$, then it has volume
  asymptotics of order $(n+2)/\al$. Any non--zero section of $\Cal
  E(\al)$, which is parallel for $\nabla$ extends by $0$ to a defining
  density for the boundary $\partial M$.

(ii) Suppose that $\barm$ is endowed with a projective structure, and
  that $\si\in\Ga(\Cal E(\al))$ is a defining density for $\partial
  M$. Then one can view $\si$ as a scale for the restriction of the
  projective structure to $M$ and the affine connection $\nabla$ on
  $M$ determined by this scale is projectively compact of order $\al$.
\end{prop}
\begin{proof}
(i) Fix a local defining function $\rho:U\to\Bbb R_{\geq 0}$ for
  $\partial M$ and let $\si$ be a nonzero section of $\Cal E(\al)\to
  M$ which is parallel for $\nabla$. Then we consider the section
  $\hat\si:=\si/\rho$ of $\Cal E(\al)$ which is defined and nowhere
  vanishing on $U\cap M$. By assumption, the connection
  $\hat\nabla=\nabla+\frac{d\rho}{\al\rho}$ extends to the
  boundary. Then from the definition of $\hat\si$, formula
  \eqref{denstrans} and $\nabla\si=0$ we get
$$
\nabla\hat\si=\nabla\tfrac{\si}{\rho}=-\tfrac{d\rho}{\rho^2}\si+
\tfrac{1}{\rho}\nabla\si=-\tfrac{d\rho}{\rho}\hat\si. 
$$ Hence the non--zero section $\hat\si$ is parallel for $\hat\nabla$
over $U\cap M$, so it extends to a parallel section for $\hat\nabla$
on all of $U$, which is nowhere vanishing. But then $\si=\hat\si\rho$
shows that $\si$ extends by zero to a defining density. The statement
on volume asymptotics then follows immediately by forming powers of
order $-(n+2)/\al$.

(ii) As a defining density for $\partial M$, $\si$ is nowhere vanishing
on $M$ and thus determines a connection $\nabla$ in the projective
class there. For a point $x_0\in\partial M$ choose an open
neighborhood $U$ and a nowhere vanishing section $\hat\si$ of $\Cal
E(\al)$ defined over $U$. Let $\hat\nabla$ be the unique connection in
the projective class on $U$ such that $\hat\nabla\hat\si=0$. Since
$\si$ is a defining density for $\partial M$, there is a defining
function $\rho:U\to\Bbb R_{\geq 0}$ for $\partial M$ such that
$\si=\hat\si\rho$. Since $\hat\si$ is parallel for $\hat\nabla$, we
get $\hat\nabla\si=\hat\si d\rho=\si \frac{d\rho}{\rho}$ over $U\cap
M$. But then \eqref{denstrans} shows that $\si$ is parallel on $U\cap
M$ for the connection $\hat\nabla+\frac{-d\rho}{\al\rho}$, which thus
has to coincide with $\nabla$.
\end{proof}

\subsection{Completeness}\label{2.1a}
Next we derive a result related to (geodesic) completeness of an
affine connection $\nabla$ which is projectively compact of some order
$\al$. By the definition of this property, the projective class of
such a connection $\nabla$ extends to all of $\barm$. In particular,
we have distinguished paths on all of $\barm$ which are the geodesic
paths of the connections in the class. We will show that, provided $\al\leq
2$, paths approaching the boundary $\partial M$ transversally do not
reach the boundary in finite time when parameterized as geodesics for
$\nabla$. Motivated by this result, we will restrict to the case
$\al\leq 2$ from now on.

\begin{prop}\label{prop2.1a}
Let $\nabla$ be an affine connection on $M$ which is projectively
compact of some order $\al\leq 2$. Suppose that one has a geodesic
path which reaches $\partial M$ in a point $x_0$ with tangent
transversal to $\partial M$. Then a part of this path can be
parameterized as a geodesic for $\nabla$ in the form
$c:[0,\infty)\to\barm$ in such a way that $c([0,\infty))\subset M$ and
    $\lim_{t\to\infty}c(t)=x_0$.
\end{prop}
\begin{proof}
  Fix a defining function $\rho$ for $\partial M$ on a neighborhood of
  $x_0$ and consider the connection
  $\hat\nabla:=\nabla+\tfrac{d\rho}{\alpha\rho}$ which is projectively
  related to $\nabla$ and extends to $\barm$. Then there is a unique
  vector $\xi$ tangent to the path at $x_0$ such that
  $d\rho(\xi)=1$. Now we can consider the (parameterized) geodesic
  $\hat c$ for $\hat\nabla$ emanating from $x_0$ in the direction $\xi$. For
  sufficiently small times, we will have $d\rho(\hat c'(t))>1/2$ and
  we restrict to an interval on which this is true. Then run along
  this curve backwards, and call the result again $\hat c$. So we may
  assume that $\hat c:[0,t_0]\to\barm$ is a geodesic for $\hat\nabla$
  such that $\hat c([0,t_0))\subset M$, $\hat c(t_0)=x_0\in\partial M$
    and $d\rho(\hat c'(t))<-1/2$ for all $t\in [0,t_0]$. In
    particular, this implies that $f:=\tfrac1{a}\rho\o\hat c$, where
    $a=\rho(\hat c(0))$, is an orientation reversing
    diffeomorphism from $[0,t_0]$ onto $[0,1]$.

  Now we know that this curve can be reparametrized as a geodesic for
  $\nabla$ in the form $c:=\hat c\o\ph$ where $\ph$ is a strictly
  increasing map defined on $[0,b)$ for an (initially unspecified)
    number $b\in\Bbb R$ such that $\ph(0)=0$ and $\ph'(0)=1$. The
    relation between the connections $\nabla$ and $\hat\nabla$
    together with the fact that $\hat c$ is a geodesic for
    $\hat\nabla$ shows that $\ph$ has to satisfy a differential
    equation. First we get
$$
  0=\nabla_{c'}c'(t)=\hat\nabla_{c'}c'(t)-
2\frac{d\rho(c'(t))}{\al\rho(c(t))}c'(t).
$$ Now we insert $c'(t)=\hat c'(\ph(t))\ph'(t)$ and use that
$c'\cdot\ph'=\ph''$ and hence $\hat c'\cdot\ph'=\ph''/\ph'$, and that
$\hat\nabla_{\hat c'}\hat c'=0$. Using the definition of $f$, we
conclude that
$$
0=\ph''(t)-2\frac{f'(\ph(t))\ph'(t)^2}{\al f(\ph(t))}.
$$ Dividing by $\ph'(t)$ (which is strictly positive), we get an
equality of logarithmic derivatives, which implies that
$\ph'(t)=C(f\o\ph)(t)^{2/\al}$ for some non--zero constant $C$. Since
we require $\ph'(0)=1$ we actually get $C=1$.

Now $f\o\ph$ is an orientation reversing diffeomorphism $[0,b)\to
  (d,1]$ where $d=f(\ph(b))$, and we use the equation on $\ph$ we have
just derived to obtain a differential equation on
$\ps:=(f\o\ph)^{-1}$. We get $\ps'(t)=1/(f\o\ph)'((f\o\ph)^{-1}(t))$,
and inserting the differential equation on $\ph$ we conclude that
$$ \ps'(t)=\frac1{t^{2/\al}f'(f^{-1}(t))}.
$$
By our assumptions $f'(f^{-1}(t))$ is strictly negative and bounded
away from zero, which shows that there are positive numbers $A<B$ such
that
$B\frac{-1}{t^{2/\al}}<\ps'(t)<A\frac{-1}{t^{2/\al}}$. Integrating,
and using $\ps(1)=0$, we conclude that for $\al<2$, we obtain 
$$A(-1+2/\al)(-1+t^{1-2/\al})\leq\ps(t)\leq B(-1+2/\al)(-1+t^{1-2/\al}),
$$ 
while for $\al=2$, we get 
$$
A(-\log(t))\leq \ps(t)\leq B(-\log(t)). 
$$ In any case, this implies that $\ps$ will be defined on $(0,1]$
with $\lim_{t\to 0}\ps(t)=\infty$, which implies our claims.
\end{proof}

\subsection{Projectively compact pseudo--Riemannian
  metrics}\label{2.3} 
The notion of being projectively compact of some order for affine
connections introduced in \ref{2.1} gives rise to an evident notion
for pseudo--Riemannian metrics.

\begin{definition}\label{def2.3}
Let $\barm=M\cup\partial M$ be as in \ref{2.1}.  For a positive real
number $\al$, a pseudo--Riemannian metric $g$ on $M$ is called
\textit{projectively compact} of order $\al$ if and only if its
Levi--Civita connection $\nabla$ is projectively compact of order
$\al$ in the sense of Definition \ref{def2.1}.
\end{definition}
Observe the volume density $\vol(g)$ of a pseudo--Riemannian metric
$g$ is parallel for the Levi--Civita connection. Hence we are always
in the setting of special affine connections and have a canonical
parallel density of each projective weight (not just up to a constant
factor). In particular, Proposition \ref{prop2.2} always applies and shows
that $\vol(g)^{-\al/(n+2)}$ extends to a defining density for the
boundary $\partial M$.

\medskip

We next prove that a certain asymptotic form of a metric implies
projective compactness of order $\al$ for many values of $\al\leq
2$. We formulate this asymptotic behavior in a form which does not
depend on a choice of coordinates. Indeed, consider an open subset
$U\subset\barm$ and a local defining function $\rho:U\to\Bbb R_{\geq
  0}$ for $\partial M$. Then for a nowhere vanishing smooth function
$C:U\to\Bbb R$ consider the $\binom{0}{2}$--tensor field $h=h_C$ on
$U\cap M$ defined by
\begin{equation}\label{expansion}
h:=\rho^{2/\al} g-C\frac{d\rho\odot d\rho}{\rho^{2/\al}}.
\end{equation}
We will assume that, locally around each point in the boundary we can
find a defining function $\rho$ and a function $C$, which satisfies
certain growth conditions towards the boundary, such that the tensor
field $h$ defined by \eqref{expansion} extends smoothly to the
boundary and that its restriction to the boundary is non--degenerate
on $T(U\cap\partial M)\subset T(U)|_{\partial M}$.

The condition just given means that, for an appropriate choice of a
function $C$ and a defining function $\rho$, we can write the metric as
\begin{equation}\label{asymptotics}
g=\frac{h}{\rho^{2/\al}}+C\frac{d\rho\odot d\rho}{\rho^{4/\al}}
\end{equation}
 with $h$ going to the boundary and restricting to a
 pseudo--Riemannian metric there. Specialized to this setting, our
 completeness result in Proposition \ref{prop2.1a} is nicely
 compatible with the result on completeness of such metrics in the
 case that $C$ is constant, see \cite{Melrose}.

It should be noted at this point that the dependence of such a form
on the defining function $\rho$ is different for different values of
$\al$. For $\al\neq 2$, one can absorb a constant factor in $C$ into a
constant rescaling of $\rho$. Moreover, if $\al<2$ the question
whether $h$ defined by \eqref{expansion} extends to the boundary
heavily depends on the defining function. Indeed if this works for
a defining function $\rho$, then this defining function is uniquely determined
up to addition of terms of the order of $\rho^2$.

In contrast, in the case $\al=2$, the condition is independent
of the choice of defining function. To see this, consider
$\hat\rho=e^f\rho$ for a smooth function $f:U\to\Bbb R$. Then
$d\hat\rho=e^fd\rho+\hat\rho df$ and thus
$\frac{d\hat\rho}{\hat\rho}=\frac{d\rho}{\rho}+df$. Forming the
symmetric product with $d\hat\rho$, one immediately concludes that the
tensor field $\hat h$ constructed from $\hat\rho$ according to
\eqref{expansion} (with $\al=2$) is given by
$$
\hat h=e^fh-2Ce^fd\rho\odot df-C\hat\rho df\odot df. 
$$ 
Evidently the last two summands are smooth up to the
boundary. Thus, smoothness of $h$ up to the boundary implies
smoothness of $\hat h$ up to the boundary. Moreover,
$$ \hat h|_{\partial M}=e^fh|_{\partial M}-2Ce^fd\rho\odot
df|_{\partial M}.
$$ 
Since the last term involves $d\rho$, it vanishes on $T(U\cap\partial
M)$, so there the two bilinear forms are conformally related. In
particular, we see that for a metric $g$, which has an asymptotic form
as in \eqref{asymptotics} with $\al=2$ and $h$ non--degenerate, the
restriction of $h$ to $T\partial M$ defines a conformal structure on
$\partial M$ which only depends on $g$ and not on the defining
function $\rho$.

\begin{thm}\label{thm2.3} 
Fix $\al\in (0,2]$ such that $\frac2\alpha$ is an integer. Suppose
  that $g$ is a pseudo--Riemannian metric on $M$ such that for each
  point $x_0\in\partial M$, we can find an open neighborhood $U$ of
  $x_0$ in $M$, a defining function $\rho:U\to\Bbb R_{\geq 0}$ for the
  boundary, and a nowhere vanishing smooth function $C:U\to\Bbb R$
  such that:
\begin{itemize}
\item For any vector field $\ze\in\frak X(U)$ with $d\rho(\ze)=0$, the
  function $\rho^{-2/\al}\ze\cdot C$ is smooth up to the boundary.
\item The tensor field $h$ defined in \eqref{expansion} extends
  smoothly to the boundary, with the restriction to the boundary being
  non--degenerate as a bilinear form on the boundary tangent bundle.
\end{itemize}
 Then $g$ is projectively compact of order $\al$.
\end{thm}
\begin{proof}
  Consider the Levi--Civita connection $\nabla$ of $g$ and the
  projectively related connection $\hat\nabla$ defined as in
  \eqref{pc-def}. To prove that $\hat\nabla$ extends to the boundary,
  it suffices to show that for arbitrary vector fields $\xi$ and
  $\eta$ defined on all of $U$, also $\hat\nabla_\xi\eta$ extends
  smoothly to all of $U$. To do this, we first show that
  $d\rho(\hat\nabla_\xi\eta)$ is smooth up to the boundary. Further we
  prove that for any smooth vector field $\ze$ on $U$ such that
  $d\rho(\ze)=0$, also $h(\hat\nabla_\xi\eta,\zeta)$ is smooth up to
  the boundary. In view of our assumptions, this evidently suffices to
  complete the proof.

We first have to construct a vector field $\ze_0$ on $U$, which plays
the role of a Reeb field. Shrinking $U$ we may assume that $d\rho$ is
nowhere vanishing on $U$, so its kernel $\ker(d\rho)$ defines a
hyperplane distribution. By assumption the restriction of $h$ to the
boundary is non--degenerate on $\ker(d\rho)$. Possibly shrinking $U$
further, we may thus assume that $h$ restricts to a non--degenerate
bilinear form on $\ker(d\rho)$ on all of $U$. This implies that we can
use $h$ to orthonormalize a local frame for $\ker(d\rho)$ and, again
shrinking, we obtain a frame $\{\xi_1,\dots\xi_n\}$ for $\ker(d\rho)$
such that $h(\xi_i,\xi_j)=\ep_i\delta_{ij}$ with $\ep_i=\pm 1$ for all
$i$. Next, choose a vector field $\tilde\zeta$ on $U$ such that
$d\rho(\tilde\zeta)=1$ and define
$\zeta_0:=\tilde\zeta-\sum_i\ep_ih(\tilde\ze,\xi_i)\xi_i$. Then it is
clear that $d\rho(\zeta_0)=d\rho(\tilde\zeta)=1$ and that for any
vector field $\xi$ on $U$ such that $d\rho(\xi)=0$, we have
$h(\zeta_0,\xi)=0$.

Now a general vector field $\xi\in\frak X(U)$ can be decomposed as
$$ 
\xi=d\rho(\xi)\ze_0+(\xi-d\rho(\xi)\ze_0).
$$ 
Since the second summand lies in $\ker(d\rho)$, we see that
$h(\ze_0,\xi)=d\rho(\xi)h(\ze_0,\ze_0)$. Now for arbitrary vector
fields $\xi,\eta\in\frak X(U)$ we can use \eqref{asymptotics} to
compute on $U\cap M$ as follows:
$$
g(\hat\nabla_\xi\eta,\ze_0)=\tfrac{1}{\rho^{2/\al}}h(\hat\nabla_\xi\eta,\ze_0)+
\tfrac{C}{\rho^{4/\al}}d\rho(\hat\nabla_\xi\eta)=d\rho(\hat\nabla_\xi\eta)
(\tfrac{C}{\rho^{4/\al}}+\tfrac{h(\ze_0,\ze_0)}{\rho^{2/\al}}).
$$ 
Observe that since $2/\alpha$ is assumed to be an integer, this
expression contains only integral powers of $\rho$, which will also
hold in the further computations in this proof. Thus positive powers
of $\rho$ will always be smooth up to the boundary with boundary value
zero. In particular, we see that smoothness of
$d\rho(\hat\nabla_\xi\eta)$ up to the boundary will follow from
smoothness of $\rho^{4/\al}g(\hat\nabla_\xi\eta,\ze_0)$ up to the
boundary.

On the other hand, consider any vector field $\zeta\in\frak X(U)$ such
that $d\rho(\zeta)=0$. Again using \eqref{asymptotics}, we obtain
$$
g(\hat\nabla_\xi\eta,\zeta)=\tfrac{1}{\rho^{2/\al}}h(\hat\nabla_\xi\eta,\zeta), 
$$ 
so we can prove smoothness of $h(\hat\nabla_\xi\eta,\zeta)$ up to the boundary
by showing that $\rho^{2/\al} g(\hat\nabla_\xi\eta,\zeta)$ is smooth 
up to the boundary.

Next the Koszul formula for the Levi--Civita connection reads as 
\begin{align*}
2g(\nabla_\xi\eta,\zeta)&=\xi\cdot g(\eta,\zeta)-\zeta\cdot
g(\xi,\eta)+\eta\cdot
g(\xi,\zeta)\\ &+g([\xi,\eta],\zeta)-g([\xi,\zeta],\eta)-g([\eta,\zeta],\xi).
\end{align*}
To compute $2g(\hat\nabla_\xi\eta,\zeta)$, we have to add
\begin{equation}\label{addon}
\frac{2d\rho(\xi)}{\al\rho}g(\eta,\zeta)+
\frac{2d\rho(\eta)}{\al\rho}g(\xi,\zeta)
\end{equation}
to the right hand side. 

Now let us first look at the case $\zeta=\zeta_0$, so we want to prove
that $\rho^{4/\al}g(\hat\nabla_\xi\eta,\zeta_0)$ is smooth up to the
boundary. Now as above, we compute on $U\cap M$:
$$
g(\eta,\ze_0)=d\rho(\eta)(\tfrac{1}{\rho^{2/\al}}h(\ze_0,\ze_0)+
\tfrac{C}{\rho^{4/\al}}).   
$$ 
Differentiating this by $\xi$, and ignoring terms which are smooth
up to the boundary after multiplication by $\rho^{4/\al}$, we are left
with $\tfrac{-4Cd\rho(\xi)d\rho(\eta)}{\al\rho^{(4+\al)/\al}}$. The
same contribution comes from $\eta\cdot g(\xi,\zeta_0)$. Next,
$$
-g(\xi,\eta)=\tfrac{-1}{\rho^{2/\al}}h(\xi,\eta)+
\tfrac{-C}{\rho^{4/\al}}d\rho(\xi)d\rho(\eta).
$$
Differentiating in direction $\ze_0$ and ignoring terms which are
smooth up to the boundary after multiplication by $\rho^{4/\al}$, we see
that this term contributes
$\tfrac{4Cd\rho(\xi)d\rho(\eta)}{\al\rho^{(4+\al)/\al}}$. 

Smoothness of $h$ and $C$ up to the boundary immediately implies that
inserting two vector fields defined on all of $U$ into $g$, and
multiplying by $\rho^{4/\al}$, the result is always smooth up to the
boundary, so there comes no further contribution from the Koszul
formula. But then
$$
\tfrac{2d\rho(\xi)}{\al\rho}g(\eta,\zeta_0)=
\tfrac{2Cd\rho(\xi)d\rho(\eta)}{\al\rho^{(4+\al)/\al}}
$$ 
up to terms which are smooth up to the boundary after
multiplication by $\rho^{4/\al}$, and the other term from
\eqref{addon} gives the same contribution. Thus we have verified that
$d\rho(\hat\nabla_\xi\eta)$ is smooth up to the boundary.

So let us turn to the case that $d\rho(\zeta)=0$, and we have to show
that $\rho^{2/\al} g(\hat\nabla_\xi\eta,\zeta)$ is smooth up to the
boundary. Notice first that $dd\rho=0$ and $d\rho(\ze)=0$ imply that
$\zeta\cdot d\rho(\xi)=d\rho([\ze,\xi])$ and likewise for
$\eta$. Together with $\zeta\cdot\rho=0$ and the fact that
$\tfrac{1}{\rho^{2/\al}}\ze\cdot C$ is smooth up to the boundary, this
easily implies that 
$$
\ze\cdot g(\xi,\eta)+g([\xi,\zeta],\eta)+g([\eta,\zeta],\xi)
$$ 
is smooth up to the boundary after multiplication by
$\rho^{2/\al}$. On the other hand, $g([\xi,\eta],\zeta)$ evidently is
smooth up to the boundary after multiplication by $\rho^{2/\al}$. For
the remaining two summands in the Koszul formula, one easily computes
that, up to terms which are smooth up to the boundary after
multiplication by $\rho^{2/\al}$, one gets
$$
-\tfrac{2d\rho(\xi)h(\eta,\zeta)}{\al\rho^{(2+\al)/\al}}
-\tfrac{2d\rho(\eta)h(\xi,\zeta)}{\al\rho^{(2+\al)/\al}}, 
$$ 
and this exactly cancels with the contribution from the two terms
in \eqref{addon}.
\end{proof}

\begin{remark}\label{rem2.3} 
(i) If $2/\alpha$ is not an integer, then our proof shows continuity
  respectively some finite order of differentiability of $\hat\nabla$
  up to the boundary.

  (ii) Notice that for $\al=1$ and $C=1$ the asymptotic form in
  \eqref{asymptotics} is called a ``scattering metric'' in
  \cite[Chapter 6]{Melrose}, where it is used to develop
  generalizations of Euclidean scattering theory.  For the more
  general asymptotic form
  $g=\tfrac{d\rho^2}{\rho^{2a}}+\tfrac{h}{\rho^{2b}}$ considered in
  \cite[Chapter 8]{Melrose}, our proof suggests that projective
  compactness forces, at least in a certain range, a relation between
  $a$ and $b$. Together with the appropriate volume asymptotics this
  then pins down both exponents.

(iii) The standard form of a conformally compact metric is
  $g=\tfrac{h+dr^2}{r^2}$ for a defining function $r$. Looking at the
  volume density, we see that this has volume asymptotics of order
  $n+1=\dim(\barm)$ as compared to $\tfrac{n+2}{\al}$ for projectively
  compact metrics of order $\al$. This indicates a significant
  difference between the two types of compactifications. In the domain
  of a local chart around the boundary, there is a way to formally
  relate compactifications of the two types. Looking at a metric of
  the form \eqref{asymptotics} with $\al=2$ and $C$ constant, one can
  formally put $\rho=r^2$ (which makes it impossible for $r$ to be a
  defining function without changing the smooth structure), to obtain
$$
g=\tfrac{h}{\rho}+C\tfrac{d\rho^2}{\rho^2}=\tfrac{h+4Cdr^2}{r^2}. 
$$ 
While this can be used to reduce some local analytical questions on
metrics which are projectively compact of order two to the conformally
compact case, it is not clear to what extent this relation has
geometric meaning. Moreover, it seems not to be possible to apply a
similar idea to metrics which are projectively compact of orders
different from $2$.
\end{remark}

%%%%%%%%%%%%%%%%%%%%%%%%%%%%%%%%%%%%%%%%%%%%%%%%%%

\section{First BGG--equations and reductions of projective
  holonomy}\label{3} 

In this section, we further explore the projective geometry of special
affine connections and in particular of pseudo--Riemannian metrics,
which are projectively compact of order one or two. By Proposition
\ref{prop2.2}, in these cases, we have a canonical defining density, which
is a section of $\Cal E(1)$ respectively of $\Cal E(2)$. Sections of
each of these two bundles form the domain of a canonical projectively
invariant overdetermined system of PDEs. These systems are coming from
the machinery of BGG sequences, which also leads to a special class of
solutions, called normal solutions. Hence we can single out
particularly nice subclasses of projectively compact affine
connections and metrics by requiring that the canonical defining
densities are solutions respectively normal solutions of the first BGG
equation. We will analyze the meaning of these conditions, noting that
for normal solutions one obtains a reduction of projective holonomy as
discussed in \cite{ageom} and \cite{hol-red}.

\subsection{Background on projective tractor bundles and first BGG
  operators}\label{3.1}

We start with a brief review on some elements of the geometry of
projective structures and a related class of projectively invariant
differential operators. We will restrict our attention to the cases we
need in this article and hence only discuss the \textit{projective
  standard cotractor bundle} and its symmetric square. More
information can be found in \cite{ageom}. The standard
cotractor bundle (or rather its dual) was introduced by T.~Thomas in
the 1930's as an alternative to the canonical Cartan connection
associated to a projective structure, a modern presentation can be
found in \cite{BEG}.

Given a manifold $N$ with a projective structure, one can define the
cotractor bundle $\Cal T^*$ simply as the first jet--prolongation
$J^1(\Cal E(1))$ of the density bundle $\Cal E(1)$ introduced in
Section \ref{2.2}. In particular, one has the jet exact sequence
\begin{equation}\label{Xdef}
0\to T^*N\otimes\Cal E(1)\to \Cal T^*\stackrel{X}{\to}\Cal E(1)\to 0.
\end{equation}
From now on, we will sometimes use abstract index notion, so we
write $\Cal E^a$ for the tangent bundle and $\Cal E_a$ for the
cotangent bundle, and we will indicate a tensor product with the line
bundle $\Cal E(w)$ by adding ``$(w)$'' to the name of a bundle. In
this notation, the jet exact sequence reads as $0\to \Cal
E_a(1)\to\Cal T^*\to\Cal E(1)\to 0$.

Choosing a connection $\nabla$ from the projective class, one obtains
an induced connection on the density bundle $\Cal E(1)$ and thus a
splitting of the jet exact sequence, i.e.~an isomorphism $\Cal
T^*\cong \Cal E_a(1)\oplus\Cal E(1)$. In this picture, we write
sections of $\Cal T^*$ as vectors $\binom{\si}{\mu_a}$ with
$\si\in\Ga(\Cal E(1))$ and $\mu_a\in\Ga(\Cal E_a(1))$. Changing from the
connection $\nabla$ to the connection $\nabla+\Up_a$ as defined in
Section \ref{2.1}, this identification changes as 
\begin{equation}
  \label{eq:std-trans}
 \binom{\si}{\mu_a}\mapsto\binom{\si}{\mu_a+\Up_a\si}, 
\end{equation}
which also shows that the projection onto the top slot and the
inclusion of the bottom slot are natural bundle maps.

There is a natural connection $\nabla^{\Cal T^*}$ on $\Cal T^*$,
which, in the identification $\Cal T^*\cong \Cal E_a(1)\oplus\Cal
E(1)$ defined by $\nabla$, is given by
\begin{equation}
  \label{eq:std-conn}
  \nabla^{\Cal T^*}_a\binom{\si}{\mu_a}=
\binom{\nabla_a\si-\mu_a}{\nabla_a\mu_b+\Rho_{ab}\si}.
\end{equation}
Here $\Rho_{ab}$ denotes the (projective) Schouten--tensor. For our
purposes it suffices to know that for a special affine connection in
the projective class, we have
$\Rho_{ab}=\tfrac{1}{\dim(N)-1}\Ric_{ab}$, where $\Ric_{ab}$ is the
usual Ricci tensor, see \cite{BEG}, and this is symmetric.

\medskip

The description of the cotractor bundle is particularly simple for the
\textit{homogeneous model} of projective geometry. We consider here the
model for orientable projective structures, which is the sphere
$S^{n+1}$ viewed as the ray projectivization of $\Bbb
R^{n+2}\setminus\{0\}$. In this case, one can actually identify $\Bbb
R^{n+2}\setminus\{0\}$ with the frame bundle of $\Cal E(1)$, so
densities of weight $w$ can be identified with smooth functions on
$\Bbb R^{n+2}\setminus\{0\}$ which are homogeneous of degree
$w$. Moreover, a local nowhere--vanishing section of $\Cal E(1)$ can
be viewed as a local section of the ray projectivization, and the
connection on $S^{n+1}$ leaving that scale parallel is just the
pullback of the standard flat connection on $\Bbb
R^{n+2}\setminus\{0\}$ along the section. 

Sections of the standard cotractor bundle $\Cal T^*$ can be identified
with one--forms on $\Bbb R^{n+2}\setminus\{0\}$ which are homogeneous
of degree $1$, and the tractor connection is again induced from the
flat connection. In particular, a parallel section of $\Cal T^*$ is
equivalent to a fixed element of $\Bbb R^{(n+2)*}$, which is viewed as
a differential form on $\Bbb R^{n+2}\setminus\{0\}$. Sections of more
general tractor bundles can be dealt with in a similar way. 

\medskip

Apart from the tractor connection, we will need a second ingredient,
the \textit{Kostant codifferential}. There is an obvious natural
bundle map $\partial^*:\Cal E_a\otimes\Cal T^*\to\Cal T^*$ defined by
$\ph_a\otimes\binom\si{\mu_a}\mapsto\binom0{\si\ph_a}$. This can be
interpreted as defining an action of the bundle $\Cal E_a$ of abelian
Lie algebras on the bundle $\Cal T^*$. Thus it extends to a sequence
of bundle maps $\partial^*:\La^kT^*N\otimes\Cal
T^*\to\La^{k-1}T^*N\otimes\Cal T^*$ such that
$\partial^*\o\partial^*=0$. Hence we have bundle maps on the bundles
of $\Cal T^*$--valued differential forms such that
$\im(\partial^*)\subset\ker(\partial^*)\subset \La^kT^*N\otimes\Cal
T^*$. What we really need is the explicit description of these two
subspaces in the case $k=1$. Using the obvious extension of the vector
notation from above, the end of the $\partial^*$--sequence has the
form
$$
\binom{\Cal E(1)}{\Cal E_a(1)}\overset{\partial^*}{\longleftarrow} 
\binom{\Cal E_a(1)}{\Cal
  E_{ab}(1)}\overset{\partial^*}{\longleftarrow} 
\binom{\Cal E_{[ab]}(1)}{\Cal E_{[ab]c}(1)}
\overset{\partial^*}{\longleftarrow}\dots 
$$
From the definition above it is evident that $\partial^*$ always maps a
row in some column to one row below in the next column.  Moreover, one
may use general tools to show that the cohomology of the sequence is
given by $\Cal E(1)$ in degree $0$ and by $\Cal E_{(ab)}$ in degree
one. This implies that $\im(\partial^*)\subset\ker(\partial^*)\subset
T^*N\otimes\Cal T^*$ has the form
\begin{equation}
  \label{eq:std-subb}
  \binom{0}{\Cal E_{[ab]}(1)}\subset \binom{0}{\Cal E_{ab}(1)}\subset
\binom{\Cal E_a(1)}{\Cal E_{ab}(1)}.
\end{equation}

\medskip

Next, we need the analogous information for the symmetric square
$S^2\Cal T^*$. Choosing a connection from the projective class, we
evidently get an isomorphism $S^2\Cal T^*\cong \Cal
E_{(ab)}(2)\oplus\Cal E_a(2)\oplus\Cal E(2)$, and we will use a vector
notation with three components, similar to the case of $\Cal
T^*$. Passing from $\nabla$ to $\hat\nabla=\nabla+\Up$, this
identification changes as 
\begin{equation}
  \label{eq:S2-trans}
  \begin{pmatrix}  \tau\\ \nu_a \\ \rho_{ab} \end{pmatrix}\mapsto 
\begin{pmatrix}  \tau\\ \nu_a+\Up_a\tau \\ 
\rho_{ab}+2\Up_{(a}\nu_{b)}+\Up_a\Up_b\tau \end{pmatrix}.
\end{equation}

The connection on $S^2\Cal T^*$ induced by $\nabla^{\Cal T^*}$ can be
easily computed directly. It is given by
\begin{equation}
  \label{eq:S2-conn}
\nabla^{S^2\Cal T^*}_a
\begin{pmatrix}
 \tau\\ \nu_b \\ \rho_{bc} 
\end{pmatrix}=
\begin{pmatrix}
  \nabla_a\tau-2\nu_a  \\ \nabla_a\nu_b+\Rho_{ab}\tau-\rho_{ab} \\ 
\nabla_a\rho_{bc}+2\Rho_{a(b}\nu_{c)}
\end{pmatrix}.  
\end{equation}
The interpretation of $\partial^*$ as an action of the bundle $\Cal E_a$ of
abelian Lie algebras on $\Cal T^*$ readily provides a similar action
$\partial^*:\Cal E_a\otimes S^2\Cal T^*\to S^2\Cal T^*$. This then
extends to a sequence of differentials defined on the bundles of
differential forms with values in $S^2\Cal T^*$.  To understand
$\im(\partial^*)\subset\ker(\partial^*)\subset T^*N\otimes S^2\Cal
T^*$, we again write out the end of the sequence:
$$
\begin{pmatrix} \Cal E(2)\\ \Cal E_a(2) \\ \Cal
  E_{(ab)}(2)\end{pmatrix} \overset{\partial^*}{\longleftarrow}
\begin{pmatrix} \Cal E_a(2)\\ \Cal E_{ab}(2) \\ \Cal
  E_{a(bc)}(2)\end{pmatrix} \overset{\partial^*}{\longleftarrow} 
\begin{pmatrix} \Cal E_{[ab]}(2)\\ \Cal E_{[ab]c}(2) \\ \Cal
  E_{[ab](cd)}(2)\end{pmatrix}
$$ As before, application of $\partial^*$ moves down one row, and the
cohomologies of the sequence are known to be $\Cal E(2)$ in degree
zero and $\Cal E_{(abc)}(2)\subset \Cal E_{a(bc)}(2)$ in degree
one. Hence the map $\partial^*$ defined on $T^*N\otimes S^2\Cal T^*$
must map $\Cal E_a(2)$ isomorphically onto the copy of the same bundle
contained in $S^2\Cal T^*$ and $\Cal E_{ab}(2)$ onto the copy of $\Cal
E_{(ab)}(2)$ contained in that bundle. Likewise, the next map
$\partial^*$ has to map $\Cal E_{[ab]}(2)$ isomorphically onto the
copy of this bundle contained in the middle slot of $\Cal E_a\otimes
S^2\Cal T^*$ and it must map $\Cal E_{[ab]c}(2)$ onto the kernel of
the complete symmetrization in the bottom slot of this bundle. Thus we
see that $\im(\partial^*)\subset\ker(\partial^*)\subset T^*N\otimes
S^2\Cal T^*$ is given by
\begin{equation}
  \label{eq:S2-subb}
\begin{pmatrix} 0 \\ \Cal E_{[ab]}(2) \\ \Cal F(2) \end{pmatrix}
\subset \begin{pmatrix} 0 \\ \Cal E_{[ab]}(2) \\  \Cal E_{a(bc)}(2)
\end{pmatrix}\subset \begin{pmatrix} \Cal E_a(2)\\ \Cal E_{ab}(2) \\ \Cal
  E_{a(bc)}(2)\end{pmatrix},  
\end{equation}
where $\Cal F\subset \Cal E_{a(bc)}$ is the kernel of the complete
symmetrization. 

\medskip

Now we are ready to describe the relation of the structures developed
so far to the first BGG operators determined by the two bundles. The
construction of BGG sequences, see e.g.~the sketch in
\cite{deformations}, shows that given any density $\psi$ of the
appropriate weight, there is a unique section $L(\psi)$ of the tractor
bundle with $\psi$ in the top slot such that applying the tractor
connection one obtains a section of the subbundle
$\ker(\partial)^*$. Then the first BGG operator is obtained by
projecting this section to the quotient bundle
$\ker(\partial^*)/\im(\partial^*)$. In particular, $\psi$ is a
solution of the first BGG--operator if and only if the covariant
derivative of $L(\psi)$ actually is a section of
$\im(\partial^*)$. There is an obvious subclass of solutions, namely
those $\psi$, for which $L(\psi)$ actually is a parallel section of
the tractor bundle in question. These are the \textit{normal
  solutions} which are the main object of study in \cite{ageom} and
\cite{hol-red}.

\subsection{Projective compactness of order one}\label{3.2} 
We want to now treat geometric structures that are related to
projective compactness of order one.  As before, we are working on a
smooth manifold $\barm$ of dimension $n+1$ with boundary; we write $M$
for the interior of $\barm$ and $\partial M$ for the boundary.

Let us first assume that $\nabla$ is a special affine connection on
$M$, which is projectively compact of order one. Then by Proposition
\ref{prop2.2} there is a natural defining density for $\partial M$
which is a section of $\Cal E(1)$. Next, we use the machinery
developed in Section \ref{3.1} to understand the splitting operator
and the first BGG operated defined on sections of this bundle.

Given any section $\si\in\Ga(\Cal E(1))$, we first have to find a section
$s\in\Ga(\Cal T^*)$ with $\si$ in the top slot and the additional
property that $\nabla^{\Cal T^*}s$ has zero in the top slot, so then
$s=L(\si)$. From the definition \eqref{eq:std-conn} of the standard
tractor connection it is clear that
\begin{equation}\label{split1}
L(\si)=\binom{\si}{\tilde\nabla_a\si}
\end{equation}
in the splitting determined by an arbitrary connection $\tilde\nabla$ in
the projective class. Applying the tractor connection to this, we get
$\binom{0}{\tilde\nabla_a\tilde\nabla_b\si+\tilde{\sf P}_{ab}\si}$, so
the first BGG operator is given by
\begin{equation}\label{1stBGG1}
\si\mapsto \tilde\nabla_{(a}\tilde\nabla_{b)}\si+\tilde\Rho_{(ab)}\si .
\end{equation} 
The existence of the splitting operator readily leads to
information on special affine connections which are projectively
compact of order $1$:

\begin{prop}\label{prop3.2}
 Let $\nabla$ be a special affine connection on $M$ which is
 projectively compact of order $\al=1$, and let $\si\in\Ga(\Cal E(1))$
 be the canonical defining density for $\partial M$ determined by
 $\nabla$. Then we have:

  (i) The section $L(\si)\in\Ga(\Cal T^*)$ is nowhere vanishing on
 $\barm$.

  (ii) The smooth section $\Rho_{ab}\si$ of $\Cal E_{(ab)}(1)$ over $M$
 extends smoothly to $\barm$. The restriction of this section to
 $\partial M$ can be naturally viewed as a second fundamental form for
 the boundary. In particular, $\partial M\subset\bar M$ is totally
 geodesic if and only if this second fundamental form vanishes
 identically on $T\partial M\x T\partial M$.
\end{prop}
\begin{proof}
  Since $\si$ is parallel for $\nabla$, formula \eqref{split1} for the
  splitting operator shows that $L(\si)=\binom{\si}{0}$ on $M$ in the
  splitting corresponding to $\nabla$, so in particular this is
  nowhere vanishing on $M$. Now choose a local defining function
  $\rho$ for the boundary and consider the connection
  $\hat\nabla=\nabla+\frac{d\rho}{\rho}$ which locally extends to all
  of $\barm$. Then from formula \eqref{eq:std-trans} we see that
  $L(\si)=\binom{\si}{\frac{\si}{\rho}d\rho}$ in the splitting
  corresponding to $\hat\nabla$. But from the proof of Proposition
  \ref{prop2.2} we know that the density $\frac{\si}{\rho}$ is parallel
  for $\hat\nabla$ and extends to $\partial M$, so it is nowhere
  vanishing. Since $d\rho$ is non--vanishing along $\partial M$, (i)
  follows.

  (ii) We have also seen already that, in the splitting defined by
  $\nabla$ on $M$, $\nabla^{\Cal T^*}L(\si)$ has zero in the top slot
  and $\Rho_{ab}\si$ in the bottom slot. Since this is concentrated in
  the bottom slot, it is independent of the choice of splitting. Now of
  course $ \nabla^{\Cal T^*}L(\si)$ is defined on all of $\barm$ and
  the values along $\partial M$ must also lie in the subbundle $\Cal
  E_{(ab)}(1)$, which gives the required extension of $\Rho_{ab}\si$.

  Now take a local defining function $\rho$ and the corresponding
  connection $\hat\nabla$ as in the first part of the proof. From
  there we know that
  $L(\si)=\binom{\si}{\hat\nabla_a\si}=\binom{\si}{\frac{\si}{\rho}d\rho}$
  in the splitting determined by $\hat\nabla$. Applying the defining
  formula \eqref{eq:std-conn} for the tractor connection and using
  that $\frac{\si}{\rho}$ is extends to a density $\hat\si$ which is
  parallel for $\hat\nabla$, we see that $\nabla^{\Cal T^*}L(\si)$ has
  zero in the top slot and $\hat\si\hat\nabla_ad\rho+\hat\Rho_{ab}\si$
  in the bottom slot.  Since $\si$ vanishes along $\partial M$, we see
  that the extension of $\Rho_{ab}\si$ is given by
  $\hat\si\hat\nabla_ad\rho$ along $\partial M$. Since $\hat\si$ is
  nowhere vanishing, we see that inserting vector fields $\xi,\eta$
  tangent to the boundary, this gives a non--zero multiple of
  $d\rho(\hat\nabla_\xi\eta)$ which justifies the interpretation as a
  (projectively weighted) second fundamental form.
\end{proof}

Notice that in the case that the second fundamental form from part (ii)
is non--degenerate, it defines a canonical conformal structure on the
boundary.

\medskip

Next, we can analyze the meaning of the canonical defining density
being a solution, respectively normal solution, of the first BGG
operator \eqref{1stBGG1}.

\begin{thm}\label{thm3.2a}
  Let $\barm$ be a smooth manifold with boundary $\partial M$ and
  interior $M$. Let $\nabla$ be a special affine connection on $M$
  which is projectively compact of order $\al=1$, and let
  $\si\in\Ga(\Cal E(1))$ be the canonical defining density for
  $\partial M$ determined by $\nabla$.

Then $\si$ is a solution of the first BGG operator defined on $\Cal
E(1)$ if and only if the connection $\nabla$ is Ricci flat. In this
case, the boundary $\partial M$ is totally geodesic and hence inherits
a projective structure. Moreover, $\si$ automatically is a normal
solution of the first BGG operator \eqref{1stBGG1} so one is in the
situation of a reduction of projective holonomy as described in
Theorem 3.1 of \cite{ageom}.
\end{thm}
\begin{proof}
We have observed already that $L(\si)=\binom{\si}{0}$ in the splitting
determined by $\nabla$, and using $\nabla$ to write the first BGG
operator, we see that $\si$ is a solution if and only if
$\Ric_{ab}\si=0$. (Recall that $\Rho_{ab}$ is symmetric and a
non--zero multiple of $\Ric_{ab}$ since $\nabla$ is a special affine
connection.) Hence $\si$ is a solution if and only if $\nabla$ is
Ricci flat. It then follows from part (ii) of Proposition \ref{prop3.2}
that $\partial M$ is totally geodesic, which in turn implies that one
obtains an induced projective structure. Moreover, if $\nabla$ is Ricci
flat, then the formulae above immediately imply that $L(\si)$ is
parallel, so $\si$ is a normal solution.
\end{proof}

Conversely, given a projective structure on $\barm$, a Ricci flat
special connection on $M$ which lies in the projective class and does
not extend to any part of the boundary must be projectively compact of
order one:
\begin{thm}\label{thm3.2b}
Let $\barm$ be a smooth manifold with boundary $\partial M$ and
interior $M$. Suppose that $\barm$ is endowed with a projective
structure and that $\nabla$ is a connection over $M$ contained in (the
restriction of) the projective class such that 
\begin{itemize}
\item $\nabla$ is special, i.e.~preserves a non--zero section $\kappa$
of $\vol(M)$;
\item $ \Ric^\nabla =0 $;
\item $\nabla$ does not extend smoothly to any neighborhood of a
boundary point.  
\end{itemize}

Then $\nabla$ is projectively compact of order $\alpha =1$,
$\si:=\kappa^{-1/(n+2)} $ smoothly extends by zero to a defining
density for $\partial M$, and this extension satisfies the equation of
the first BGG operator \eqref{1stBGG1} on $\barm$. So we are in
the situation of Theorem \ref{thm3.2a}.
\end{thm}
\begin{proof}
  By definition, $\si$ is a section of $\mathcal{E}(1)$ defined on $M$
  and nowhere vanishing there. Since $\Ric^\nabla =0$, and $\nabla\si
  =0$ it follows that $I:=L(\si)=\binom{\si}{0}$ is parallel on
  $M$. But then $I$ extends by parallel transport to a parallel
  tractor on all of $\barm$ and projecting this to the quotient bundle
  $\Cal E(1)$, we obtain a smooth extension of $\si$ to $\barm$. By
  continuity we see that on all of $\barm$, we have $I=L(\si)$ and
  that $\si$ satisfies the first BGG equation \eqref{1stBGG1} on
  $\barm$.

  Next we claim that (the extension of) $\si$ is identically zero on
  $\partial M$. If $q\in\partial M$ would be a point such that
  $\si(q)\neq 0$, then take an open neighborhood $U$ of $q$ on which
  $\si$ is non--vanishing. Then there is a unique connection in (the
  restriction of) the projective class on $U$ for which $\si$ is
  parallel. By construction, this agrees with $\nabla$ on
  $U\cap\partial M$, thus providing an extension contradicting our
  assumptions. Hence we see that the zero locus of $\si$ coincides
  with $\partial M$. Finally, $I$ is parallel and hence nowhere zero,
  thus from \eqref{split1} we see that for any connection
  $\tilde\nabla$ in the projective class, which extends to the
  boundary, $\tilde{\nabla}\si$ is nowhere zero along $\partial
  M$. Thus $\si$ is a defining density for $\partial M$.
\end{proof}

\subsection{Projective compactness of order two}\label{3.3} 
We now want to consider structures which turn out to be related to
projective compactness of order two.

First let us assume that we have an affine connection $\nabla$ on $M$
which is projectively compact of order two,  and let us denote by
$\tau\in\Ga(\Cal E(2))$ a corresponding defining density (which is
unique up to a constant factor). To
apply the machinery from Section \ref{3.1}, we first have to describe
$L(\tau)\in\Ga(S^2\Cal T^*)$. Taking any connection $\tilde\nabla$ in
the projective class, formula \eqref{eq:S2-conn} shows that we must
have $\nu_a=\tfrac12\tilde\nabla_a\tau$, and then
$\rho_{ab}=\tfrac12\tilde\nabla_{(a}\tilde\nabla_{b)}\tau+\tilde\Rho_{(ab)}\tau$,
and this describes $L(\tau)$. In particular, in the splitting
determined by $\nabla$, we get
\begin{equation}\label{Lform}
L(\tau)=\begin{pmatrix} \tau \\ 0\\ \Rho_{ab}\tau \end{pmatrix} \qquad 
\nabla^{S^2\Cal T^*}_aL(\tau)=\begin{pmatrix} 0 \\ 0\\ 
\tau\nabla_a\Rho_{bc}\end{pmatrix},
\end{equation}
where as before we use that $\Rho_{ab}$ is symmetric. 

\begin{thm}\label{thm3.3a} 
Let $\barm$ be a smooth manifold with boundary $\partial M$ and
interior $M$. Let $\nabla$ be a special affine connection on $M$ which
is projectively compact of order two, let $\Ric_{ab}$ be its Ricci
curvature, and let $\tau\in\Cal E(2)$ be the corresponding defining
density. Then:

(i) $\tau$ is a solution of the first BGG operator if and only if
$\nabla_{(a}\Ric_{bc)}=0$. 

(ii) $\tau$ is a normal solution if and only if
$\nabla_a\Ric_{bc}=0$. Assuming further that $\Ric_{ab}$ is
non--degenerate, it defines a pseudo--Riemannian Einstein--metric on
$M$ with Levi--Civita connection $\nabla$. In this case, $L(\tau)$
defines a non--degenerate bundle metric (necessarily of indefinite
signature) on the standard tractor bundle over $\barm$. This gives
rise to a reduction of projective holonomy to an orthogonal group as
studied in Section 3.3 of \cite{ageom} and in Section 3.1 of
\cite{hol-red}, with the closed curved orbit given by the boundary
$\partial M$ and the open curved orbit given by the interior $M$. In
particular, as shown in these references, the boundary $\partial M$
inherits a conformal structure.
\end{thm}
\begin{proof}
From the formula for $\nabla^{S^2\Cal T^*}L(\tau)$ above, equation
\eqref{eq:S2-subb}, and the fact that $\Rho_{ab}=\tfrac1{n}\Ric_{ab}$,
we immediately get (i) and the first statement in (ii). If in the
setting of (ii) we assume that $\Ric_{ab}$ is non--degenerate, then it
defines a pseudo--Riemannian metric on $M$ which is parallel for
$\nabla$, which thus must be its Levi--Civita connection.  But by
construction, the Ricci curvature of this metric is $\Ric_{ab}$
itself, so we obtain Einstein metric (cf.\ \cite{Armstrong}). 

As noted in Section \ref{3.1}, associated to the choice of the
connection $\nabla$ in the projective class, there is an
identification $\Cal T^*\cong\Cal E_a(1)\oplus\Cal
E(1)$. Correspondingly, the standard tractor bundle decomposes as
$\Cal T\cong\Cal E^a(-1)\oplus\Cal E(-1)$. Viewing sections of
$S^2\Cal T^*$ as bilinear forms on $\Cal T$, the decomposition into
triples we have used has the restrictions to the two summands in the
top and bottom slots and the cross--term in the middle slot. As we
have noted above, in the splitting determined by $\nabla$, we have
$$ 
L(\tau)=\begin{pmatrix} \tau \\ 0\\ \Rho_{ab}\tau \end{pmatrix}.
$$ 

Vanishing of the middle slot shows that over $M$, the decomposition
$\Cal T=\Cal E(1)\oplus\Cal E_a(1)$ determined by $\nabla$ is
orthogonal for the bilinear form $L(\tau)$. Since $\Rho_{ab}$ is a
non--zero multiple of $\Ric_{ab}$, non--degeneracy of $\Ric_{ab}$
implies that the restriction of $L(\tau)$ to both summands is
non--degenerate. This shows that $L(\tau)$ is a non--degenerate
bilinear form on $\Cal T$ over $M$.  Since $L(\tau)$ is parallel, this
is true over all of $\barm$ and we get a holonomy reduction as
claimed. In the references mentioned in the theorem, it is shown that
the curved orbit decomposition is determined by the sign of the
density $\tau$, which implies the last part.
\end{proof}

Again, we can also prove a nice converse to this result:
\begin{thm}\label{thm3.3b}
Let $\barm$ be a smooth manifold with boundary $\partial M$ and
interior $M$. Suppose that $\barm$ is endowed with a projective
structure and that $\nabla$ is a connection in (the restriction of)
the projective class on $M$ which is the Levi--Civita connection of a
non--Ricci--flat Einstein metric or, equivalently, satisfies: 
\begin{itemize}
\item  $\nabla$ is special, i.e.~it preserves a non--zero section
$\kappa$ of $\vol(M)$;
\item $\nabla \Ric^\nabla =0 $, and $\Ric^\nabla$ is
  non-degenerate. 
\end{itemize} 
Suppose further that $\nabla$ does not
smoothly extend to any neighborhood of a boundary point.

Then $\nabla$ is projectively compact of order $\alpha =2$,
$\tau:=\kappa^{-2/(n+2)} $ smoothly extends by zero to a defining
density for $\partial M$, and this extension is a normal solution of
the first BGG--equation on $\mathcal{E}(2)$ on $\barm$. So we are
again in the situation of a holonomy reduction as in part (ii) of
Theorem \ref{thm3.3a}.
\end{thm} 
\begin{proof}
By definition, $\tau$ is a section of $\mathcal{E}(2)$ defined on $M$.
Since $\nabla \Ric^\nabla =0$, and $\nabla \tau =0$ it follows that
$H:=L(\tau)$ (as in \eqref{Lform}) is parallel on $M$. Now the
argument follows the proof of Theorem \ref{thm3.2b}, mutatis mutandis,
up to the point that the zero locus of $\tau$ coincides with $\partial
M$. To see that $\tau$ is indeed a defining density, observe that the
projection of $H$ to the quotient bundle $\Cal E(2)$ coincides with
(the extension of) $\tau$, so it vanishes along $\partial M$. In the
proof of Theorem \ref{thm3.3a} we have noted that this describes the
restriction of $H$ to the natural line subbundle $\Cal
E(-1)\subset\Cal T$. Non--degeneracy of $H$ then implies that the
middle slot of $H$ (which describes the cross--term of the bilinear
form) is nowhere vanishing along $\partial M$. Now by \eqref{Lform}
and \eqref{eq:S2-trans} this middle slot is a non--zero multiple of
$\frac{1}{2}\tilde{\nabla}\tau$, where $\tilde\nabla$ is any
connection in the projective class that extends to the boundary.
\end{proof}

Now we can analyze the section $L(\tau)$ in a similar way as for
projective compactness of order one studied in Section \ref{3.2}. This
only works in the setting of an non--Ricci--flat Einstein metric as in
part (ii) of Theorem \ref{thm3.3a}. In this case, we get a converse to
Theorem \ref{thm2.3}.

\begin{prop}\label{prop3.3}
Suppose that $g$ is a projectively compact pseudo--Riemannian Einstein
metric on $M$ with non--zero scalar curvature $R$.

Then for any local defining function $\rho$ for $\partial M$, the
symmetric $\binom02$--tensor field $\rho
g+\tfrac{n(n+1)}{4R}\frac{d\rho^2}{\rho}$ on $TM$ extends smoothly to
the boundary and its boundary value restricts to a non--degenerate
symmetric bilinear form on $T\partial M$. Thus $g$ has an asymptotic
form as in \eqref{asymptotics} with $C=\tfrac{-n(n+1)}{4R}$.
\end{prop}
\begin{proof}
From Theorem \ref{thm3.3b}, we know that the order of projective
compactness is $\al=2$ and we are in the situaiton of part (ii) of
Theorem \ref{thm3.3a}. From the proof of that theorem we see that, in
the splitting determined by $\nabla$, we have
$$ 
L(\tau)=\begin{pmatrix} \tau \\ 0\\ \tfrac{1}{n(n+1)}Rg_{ab}\tau \end{pmatrix},
$$ where we have used that
$\Rho_{ab}=\tfrac{1}{n}\Ric_{ab}=\tfrac{1}{n(n+1)}Rg_{ab}$ for an
Einstein metric. Now we compute the expression for $L(\tau)$ in the
splitting determined by the connection
$\hat\nabla=\nabla+\frac{d\rho}{2\rho}$, which extends to the
boundary. By formula \eqref{eq:S2-trans}, this is given by
$$
\begin{pmatrix} \tau
  \\ \tau\tfrac{d\rho}{2\rho}\\ 
(\tfrac{1}{n(n+1)}Rg_{ab}+\frac{d\rho^2}{4\rho^2})\tau \end{pmatrix}. 
$$ 
Of course the top slot vanishes along $\partial M$.  As we have noted
in Section \ref{2.2}, $\frac{\tau}{\rho}=:\hat\tau$ is a section of
$\Cal E(2)$, which is parallel for $\hat\nabla$, and thus is nowhere
vanishing on the domain of definition of $\rho$. Consequently, the
middle slot of this expression approaches a non--zero multiple of
$d\rho$ in each point of the boundary. The bottom slot is given by
$\hat\tau\tfrac{R}{n(n+1)}h_{ab}$, where
$$ 
h_{ab}=\rho g_{ab}+\tfrac{n(n+1)}{4R}\tfrac{d\rho^2}{\rho}
$$ 
so $h_{ab}$ has to extend to the boundary. Then the fact that
$L(\tau)$ remains non--degenerate along the boundary is equivalent to
the fact that the restriction of $h_{ab}$ to the kernel of $d\rho$ is
non--degenerate along the boundary. Since $\rho$ is a defining
function, this kernel is $T\partial M\subset T\barm|_{\partial M}$.
\end{proof}

\begin{remark}\label{rem3.3}
In the situation of a holonomy reduction as in part (ii) of Theorem
\ref{thm3.3a}, the general theory developed in \cite{hol-red} implies
that the restriction of the projective standard tractor bundle $\Cal
T$ of $\barm$ to $\partial M$ together with the bundle metric defined
by $L(\tau)$ can be identified with the conformal standard tractor
bundle for the conformal structure on $\partial M$ induced by the
holonomy reduction. Now $\Cal T|_{\partial M}$ inherits a finer
filtration from $L(\tau)$ since the distinguished line subbundle is
isotropic for $L(\tau)$ over $\partial M$ and thus contained in its
orthocomplement. It is well known that the quotient of these two
bundles is isomorphic to the tangent bundle twisted by a density
bundle and that the conformal metric coincides with the bundle metric
on this quotient induced by the tractor metric. Now the proof above
also shows that this induced metric in some scale is represented by
the restriction of a non--zero multiple of $h_{ab}$ to $T\partial
M$. This shows that in the case of a holonomy reduction, the conformal
structure on $\partial M$ induced by the asymptotic form obtained in
Proposition \ref{prop3.3} (see the discussion before Theorem
\ref{thm2.3}) coincides with the one induced by the holonomy
reduction.
\end{remark}

\subsection{The case of Ricci--flat metrics}\label{3.4}
Suppose that we have given a Ricci--flat pseudo--Riemannian metric $g$
on $M\subset\barm$, let $\nabla$ be its Levi--Civita connection and
consider $\si:=\vol(g)^{-1/(n+2)}\in\Ga(\Cal E(1))$. Then by Theorem
\ref{thm3.2a}, $\si$ satisfies the first BGG equation defined by
\eqref{1stBGG1}. Likewise, a slight variant of Theorem \ref{thm3.3a}
shows that $\tau : =\si^2\in\Ga(\Cal E(2))$ has to be a normal solution
of the first BGG equation defined on the bundle $\Cal E(2)$. This is
easy to explain: From the proof of Theorem \ref{thm3.2a}, we see that
$\si$ is automatically a normal solution, so $s=L(\si)$ is a parallel
section of $\Cal T^*$. But then $s\odot s$ is a parallel section of
the tractor bundle $S^2\Cal T^*$ and thus determines a normal solution
of the corresponding first BGG equation, which is clearly given by
$\si^2$.

\begin{remark}\label{rem3.4a} Of course $\tau = \si^2$ cannot be a boundary 
defining density. Nevertheless the other observations suggest that
perhaps there could be two natural notions of a projective
compactification for a Ricci flat metric, corresponding to projective
compactness of order one and order two, respectively.  However by
Theorem \ref{thm3.2b} a projectively compact Ricci flat metric is
necessarily of order $\alpha =1$. That a Ricci--flat metric $g$ on $M$
cannot be projectively compact of order two can also be seen directly
as follows.

Assuming that $g$ is projectively compact of order two, consider the
natural defining density $\tau:=\vol(g)^{-2/(n+2)}\in\Ga(\Cal E(2))$
for $\partial M$. Then from Section \ref{3.3} we know that the section
$L(\tau)$ of $S^2(\Cal T^*)$ is parallel and in the splitting defined
by $\nabla$ it has $\tau$ in the top slot while the other two slots
are identically zero. Thus, as a bilinear form of $\Cal T$, $L(\tau)$
has rank one over $M$, and since it is parallel, this holds over all
of $\barm$. As in the proof of Proposition \ref{prop3.3} we can next
compute $L(\tau)$ in the splitting corresponding to the connection
$\hat\nabla=\nabla+\frac{d\rho}{2\rho}$ which by assumption extends to
the boundary. This is given by
$$
\begin{pmatrix} \tau
  \\ \tau\tfrac{d\rho}{2\rho}\\ \frac{d\rho^2}{4\rho^2}\tau \end{pmatrix}.
$$ But as before, $\hat\tau=\frac{\tau}{\rho}$ is parallel for
$\hat\nabla$ and thus extends to the boundary with non--zero boundary
value. The same holds for $d\rho$ and hence $d\rho^2=d\rho\odot d\rho$
extends to the boundary with non--zero boundary value. As before, the
middle slot is just $\tfrac{\hat\tau}{2}d\rho$, so this is fine, but
the bottom slot is $\tfrac{\hat\tau}{4\rho}d\rho^2$, which cannot
extend, so we obtain a contradiction.
\end{remark}

There are nice cases of Ricci flat metrics which are projectively
compact (of order one). The simplest example of this situation is
provided by the homogeneous model of projective geometry, see Section
\ref{3.1}. Consider the sphere $S^{n+1}$ as the ray projectivization
of $\Bbb R^{n+2}\setminus\{0\}$. Recall from Section \ref{3.1} that a
local scale for this projective structure is determined by a local
smooth section of the ray projectivization, and the corresponding
connection is the pullback of the flat connection. In particular, the
embedding of $S^{n+1}$ as the unit sphere of $\Bbb R^{n+2}$ is a
global section and the corresponding pullback connection is just the
Levi--Civita connection of the round metric on $S^{n+1}$. On the other
hand, we can define a local section over an open hemisphere by mapping
the round hemisphere to an affine hyperplane in $\Bbb R^{n+2}$ via
central projection. The resulting connection is then the pullback of
the flat connection on that affine hyperplane. The latter is the Levi
Civita connection for the flat metric on $\Bbb R^{n+1}$ of any chosen
signature. Moreover, the corresponding scale is just given by the
restriction of a fixed linear functional on $\Bbb R^{n+2}$, which, as
we have seen in Section \ref{3.1}, corresponds to a parallel standard
cotractor on $S^{n+1}$. In particular, we can pass to the closed
hemisphere and then the section of $\Cal E(1)$ underlying this
parallel cotractor is a defining density for the boundary sphere
$S^n$. This shows that the flat metric on a hemisphere obtained via
central projection is projectively compact of order one.

\subsection{Projectively compact Ricci flat metrics}\label{3.5} 
To proceed with the
analysis of this case we have to involve a new ingredient, namely the
so--called projective metricity equation. This is the first BGG
equation associated to the bundle $S^2\Cal T$, the dual of the tractor
bundle giving rise to the first BGG equation on $\Cal E(2)$ as studied
in Section \ref{3.3}. The relation between the first BGG equations
determined by the two bundles is much more complicated than mere
duality, however. The metricity equation is discussed in \cite{CGM} in
a way closely analogous to the discussion in Section \ref{3.1}, and we
take some information from there. The natural quotient bundle of
$S^2\Cal T$, on which the first BGG equation is defined is the bundle
$\Cal E^{(ab)}(-2)$, a weighted version of the bundle of symmetric
bilinear forms on the cotangent bundle.

The main information we need at this place concerns a manifold $N$ of
dimension $n+1$ endowed with a projective structure containing the
Levi--Civita connection of a pseudo--Riemannian metric $g$. Then
putting $\si:=\vol(g)^{-1/(n+2)}\in\Ga(\Cal E(1))$, and denoting by
$g^{-1}\in\Ga(S^2TN)$ the inverse of $g$, the section
$\si^{-2}g^{-1}\in\Ga(S^2TN(-2))$ is a solution of this first BGG
operator. (This is an easy consequence of the fact that it is parallel
for the connection $\nabla$ from the projective class and the BGG
operator in this case is of order one.)  In \cite{CGM} it is shown
that this solution is normal if and only if $g$ is Einstein, so in
that case $L(\si^{-2}g^{-1})$ is a parallel section of $S^2\Cal T$,
and hence can be interpreted as a parallel (degenerate) bundle metric
on the standard cotractor bundle $\Cal T^*$ (cf.\ \cite[Theorem
  3.1]{GoMac}).

Similarly as in Section \ref{3.3} it is easy to describe
$L(\si^{-2}g^{-1})$ in the splitting determined by $\nabla$. If $g$ is
Ricci flat (indeed, scalar flat is sufficient for this), then it has
$\si^{-2}g^{-1}$ in the projecting slot and $0$ in both other
slots. This immediately implies that, as a bilinear form on $\Cal
T^*$, the section $L(\si^{-2}g^{-1})$ has (constant) rank $n+1$
(i.e.~corank one). Moreover, from Section \ref{3.2}, we see that the
parallel section $L(\si)\in\Ga(\Cal T^*)$ corresponding to $\si$ is
concentrated in the projecting slot, which immediately implies that it
spans the null--space of the degenerate bilinear tractor form
$L(\si^{-2}g^{-1})$.

Returning to our usual setting, these observations suffice to describe
the structure on the boundary induced by a projectively compact Ricci
flat metric in the interior.

\begin{thm}\label{thm3.5}
  Let $\barm$ be a smooth manifold of dimension $n+1$ with boundary
  $\partial M$ and interior $M$, and suppose that $g$ is a
  projectively compact Ricci flat pseudo--Riemannian metric of
  signature $(p,q)$ on $M$. Then the order of projective compactness
  is one and the induced projective structure on $\partial M$, from
  Theorem \ref{thm3.2a}, canonically inherits a holonomy reduction to
  the group $SO(p,q)\subset SL(n+1,\Bbb R)$.
\end{thm}
\begin{proof}
  By Theorem \ref{thm3.2b} the metric is projectively compact of order
  one and $\si:=\vol(g)^{-1/(n+2)}\in\Ga(\Cal E(1))$ is a defining
  density for $\partial M$. Consider the solution $\si^{-2}g^{-1}$ of
  the metricity equation on $M$. As discussed above, the corresponding
  section $L(\si^{-2}g^{-1})$ of $S^2\Cal T$ is parallel over $M$, so
  since the projective structure extends to the boundary, it extends
  to a parallel section over all of $\barm$. As a bilinear form on
  $\Cal T^*$, $L(\si^{-2}g^{-1})$ has rank $n+1$ over $M$, so this
  also holds on the boundary. Moreover, the parallel section
  $L(\si)\in\Ga(\Cal T^*)$ spans the null space of $L(\si^{-2}g^{-1})$
  over $M$, and again this continues to hold over $\barm$. Finally, we
  know from the proof of Proposition \ref{prop3.2} that
  $L(\si)=\binom{0}{\frac{\si}{\rho}d\rho}$ along the boundary.

Using this, we can now nicely describe the induced projective
structure on $\partial M$ as the kernel of $L(\si)$. Indeed, since
$L(\si)$ is nowhere vanishing, its kernel defines a smooth corank one
subbundle $\tilde{\Cal T}\subset \Cal T|_{\partial M}$. Moreover, along
the boundary, $L(\si)$ defines a section of $T^*M\otimes\Cal E(1)\cong
T^*M\otimes\vol(\barm)^{-1/(n+2)}$ whose pointwise kernel is
$T(\partial M)\otimes\vol(\barm)^{-1/(n+2)}$. Denoting by $\Cal N$ the
conormal bundle of the boundary, we have obtained a section of $\Cal
N\otimes\vol(\barm)^{-1/(n+2)}$. But of course,
$\vol(\barm)|_{\partial M}\cong\Cal N\otimes\vol(\partial M)$, so
$\Cal N\cong\vol(\partial M)^{-1}\otimes\vol(\barm)|_{\partial M}$,
and we can interpret $L(\si)$ as a non--vanishing section of
$\vol(\partial M)^{-1}\otimes\vol(\barm)^{(n+1)/(n+2)}|_{\partial
  M}$. This section identifies $\vol(\partial M)$ with
$\vol(\barm)^{(n+1)/(n+2)}|_{\partial M}$. Taking the power of this
of order $-1/(n+1)$ we obtain an isomorphism of $\Cal E(1)|_{\partial
  M}$ with the space of densities of projective weight one on
$\partial M$.

Thus we conclude that $\tilde{\Cal T}\to\partial M$ is a bundle of
rank $n+1$ which inherits the appropriate composition series for a
projective standard tractor bundle. Since $L(\si)$ is parallel, the
standard tractor connection on $\Cal T$ restricts to a connection on
the vector bundle $\tilde{\Cal T}$, and in \cite[Theorem 3.1]{ageom}
it is shown that this restriction is normal. Hence we can view
$\tilde{\Cal T}$ with the standard tractor bundle of the induced
projective structure on $\partial M$.

By duality, the standard cotractor bundle $\tilde{\Cal T^*}$ for this
structure can be identified with the quotient of $\Cal T^*|_{\partial
  M}$ by the line spanned by $L(\si)$. But then we know that
$L(\si^{-2}g^{-1})$ descends to a non--degenerate bundle metric on
this quotient bundle, which has the same signature as $g$ and by
construction is parallel for the induced connection. Hence the inverse
defines a non--degenerate parallel metric of signature $(p,q)$ on the
standard tractor bundle, thus giving rise to the claimed holonomy
reduction. 
\end{proof}

\begin{remark}\label{rem3.5a}
Note that the Theorem statement above could be strengthened without
adjusting the proof. Rather than requiring the Ricci-flat Levi-Civita connection
to be projectively compact it would be sufficient to assume that its
projective class extends to the boundary, while the connection itself
does not (along the lines of Theorem \ref{thm3.2b}).
\end{remark} 

Projective holonomy reductions to orthogonal groups have been studied
in detail in Section 3.2 of \cite{ageom} and in Section 3.1 of
\cite{hol-red} and we use the results obtained there. If we start with
a Riemannian metric $g$, then the reduction will be to the orthogonal
group $SO(n+1)\subset SL(n+1,\Bbb R)$ and this amounts to a positive
Einstein Riemannian metric in the projective class. If the initial
metric is pseudo--Riemannian of signature $(p,q)$ with $p,q>0$, then
the holonomy reduction induces the so--called \textit{curved orbit
  decomposition} $\partial M=\partial M_+\cup \partial M_0\cup\partial
M_-$ with $\partial M_\pm$ open in $\partial M$, while $\partial M_0$
(if non--empty) is an embedded hypersurface, which separates $\partial
M_+$ and $\partial M_-$. On $\partial M_\pm$ the holonomy reduction
determines Einstein metrics in the projective class of signature
$(p-1,q)$ and $(p,q-1)$, respectively. On $\partial M_0$, one obtains
a well defined conformal structure of signature $(p-1,q-1)$ whose
normal conformal standard tractor bundle with its canonical connection
coincides with the restriction of $\tilde{\Cal T}$. We shall see below
how to describe this decomposition explicitly.

\medskip

Let us analyze the orbit decomposition in the case of the homogeneous
model. As in Section \ref{3.4} we consider the flat connection on an
open hemisphere in $S^{n+1}$ obtained via central projection to an
affine hyperplane, and this is projectively compact (of order one) on
the closed hemisphere. The corresponding parallel standard cotractor
$I=L(\si)$ is described by the functional whose kernel projectivizes
to the boundary sphere $S^n$. Now the flat connection on an affine
hyperplane is the Levi--Civita connection of the flat metric of any
signature, and we consider a metric of Lorentzian signature
$(n,1)$. This metric is encoded as a parallel bilinear form on the
standard cotractor bundle whose null space is spanned by $I$.

In the case of the homogeneous model, this corresponds to a fixed
element $H$ of $S^2\Bbb R^{n+2}$, which has rank $n+1$ with null space
spanned by $I$. Then $H$ descends to a non--degenerate bilinear form
on $\Bbb R^{(n+2)*}/(\Bbb R I)$. This is the dual space of $\ker(I)$
so we can view the inverse $H^{-1}$ as a non--degenerate bilinear form
on $\ker(I)$. Now of course, the boundary sphere $S^n$ can be viewed
as the ray projectivization of $\ker(I)\setminus\{0\}$ and $H^{-1}$
describes a parallel section of the symmetric square of the cotractor
bundle for the resulting flat projective structure.  Now the orbit
decomposition is described in \cite{ageom} and \cite{hol-red}. It
exactly corresponds to the restriction of $H^{-1}$ to a ray being
positive definite, negative definite, and zero, respectively. So these
are just the points at infinity reached by space--like, respectively
time--like, respectively null lines in the original Lorentzian vector
space. The open curved orbits are the spaces of positive respectively
negative rays and thus the standard models for hyperbolic spaces of
the appropriate signature. The closed curved orbit consists of two
copies of the sphere $S^{n-1}$ viewed as the ray projectivized light
cone of a Lorentzian metric, so each of the two copies is a
homogeneous model of conformal geometry in Riemannian signature.

\begin{remark}\label{rem3.5b} 
  For emphasis, we point out here that the projective compactification
  of Minkowski space is very different from the usual conformal
  compactification of Minkowski space, which conformally embeds
  Minkowski space into a subspace of the Einstein cone. (This
  compactification is due to Penrose, see
  e.g. \cite{Fr,Penrose125}). Whereas here in the projective
  compactification the set of points at infinity reached by
  space--like geodesic rays is an open set, by contrast in the
  conformal compactification all such rays end in a point of the
  conformal infinity. It is similar for future time--like rays, and
  past time--like rays. They end in open caps of the boundary sphere
  $S^n$ of the projective
  compactification, while they end respectively in the two points
  known as ``future and past time-like infinity'' in the conformal
  compactification. In the projective compactification the end points
  of future directed null rays form an $S^{n-1}$ whereas in the
  conformal compactification the ``future null infinity'' is open in
  the boundary (and so of dimension $n$). Again it is similar for past
  null rays.
\end{remark}

\subsection{Explicit form of curved orbit decomposition}\label{3.6} 
To obtain an explicit description of the curved orbits in general, we
have to analyze the parallel section $L(\si^{-2}g^{-1})$ determined by
a Ricci flat metric which is projectively compact of order one in more
detail. Dual to the description of $S^2\Cal T^*$ in Section \ref{3.1},
a choice of connection in the projective class identifies the bundle
$S^2\Cal T$ with the direct sum $\Cal E(-2)\oplus\Cal
E^a(-2)\oplus\Cal E^{(ab)}(-2)$. Dualizing \eqref{eq:S2-trans}, we see
how this identification changes when passing from $\nabla_a$ to
$\hat\nabla_a=\nabla_a+\Up_a$: 
\begin{equation}
  \label{eq:S2*-trans}
\begin{pmatrix} \tau^{ab} \\ \la^a \\ \nu \end{pmatrix}\mapsto 
\begin{pmatrix} \tau^{ab} \\ \la^a-\tau^{ab}\Up_b
  \\ \nu-2\la^a\Up_a+\tau^{ab}\Up_a\Up_b\end{pmatrix}
\end{equation}

\begin{thm}\label{thm3.6}
  Let $\barm$ be a smooth manifold of dimension $n+1$ with boundary
  $\partial M$ and interior $M$, and let $g$ be a projectively compact
  Ricci--flat pseudo--Riemannian metric on $M$. Let $\rho:U\to\Bbb
  R_{\geq 0}$ be a local defining function for $\partial M$, and let
  us write $\rho_a$ for the one--form $d\rho$.

(i) The section $\rho^{-2}g^{ab}$ of $S^2TM$ extends smoothly to the
boundary and the boundary value $\tau^{ab}$ satisfies
$\tau^{ab}\rho_b=0$. Moreover, the curved orbits $\partial M_\pm$
consist of those points of $M$ in which the bilinear form $\tau^{ab}$
has rank $n$, while in points of $\partial M_0$ it has rank $n-1$.

(ii) The section $\rho^{-3}g^{ab}\rho_b$ of $TM$ extends smoothly to
the boundary, and the boundary value $\la^a$ satisfies
$\la^a\rho_a=0$. 

(iii) The function $\rho^{-4}g^{ab}\rho_a\rho_b$ on $M$ extends smoothly
to the boundary.  
\end{thm}
\begin{proof} As before, we write $\si:=\vol(g)^{-1/(n+2)}$. Then from
above we know that in the splitting determined by $\nabla$, the
parallel section $L(\si^{-2}g^{-1})$ has $\si^{-2}g^{ab}$ in the top
slot and zero in the other two slots. Now take a local defining
function $\rho:U\to\Bbb R_{\geq 0}$ for the boundary, and pass to the
splitting defined by $\hat\nabla=\nabla+\frac{d\rho}{\rho}$ which
extends to the boundary. Using \eqref{eq:S2*-trans}, we see that in
this splitting  
$$ L(\si^{-2}g^{-1})=\begin{pmatrix} \si^{-2}g^{ab}
\\ -\si^{-2}\rho^{-1}g^{ab}\rho_b
\\ \si^{-2}\rho^{-2}g^{ab}\rho_a\rho_b
\end{pmatrix}. 
$$ Since $\si/\rho$ has a finite non--zero limit to the boundary, the
same holds for $\si^{-2}\rho^2$. Pulling this out, we see that the
three slots of are (up to sign) exactly the three objects claimed to
smoothly extend to the boundary, so these claims follow. Next, we know
that $L(\si^{-2}g^{-1})$ has rank $n+1$ with null space spanned by
$L(\si)$. This implies that $\rho_a$ lies in the null space of
$\tau^{ab}$ and in the kernel of $\la^a$ (which also follows from the
existence of the limits towards the boundary shown above). Since
$\tau^{ab}$ describes the restriction of $L(\si^{-2}g^{-1})$ to a
subspace of codimension one and is degenerate, the only possible ranks
for $\tau^{ab}$ are $n$ and $n-1$. So it remains to prove the relation
to the curved orbit decomposition.

The holonomy reduction giving rise to the curved orbit decomposition
comes from the inverse of the metric induced by $L(\si^{-2}g^{-1})$ on the
quotient of $\Cal T^*$ by the line spanned by $L(\si)$. Now this
inverse is a section of the bundle $S^2\tilde{\Cal T}^*$ of metrics on
the standard cotractor bundle $\tilde{\Cal T}$ for the induced
projective structure on the boundary. The irreducible quotient of this
is the bundle of densities of projective weight two, with the quotient
projection coming from the restriction of the metric to the natural
line subbundle in $\tilde{\Cal T}$. By theorem 3.1 of \cite{hol-red},
the curved orbit $\partial M_0$ coincides with the zero set of the
induced section of the quotient bundle. Otherwise put, a point $x_0$
lies in $\partial M_0$ if and only if the distinguished line in
$\tilde{\Cal T}$ is isotropic for the metric defining the holonomy
reduction. But this is equivalent to the fact that the dual metric is
degenerate on the annihilator of this line, which is exactly the
natural subbundle in $\tilde{\Cal T}^*$. Of course, this is equivalent
to the null--space of $\tau^{ab}$ being strictly bigger than the line
spanned by $\rho_a$.
\end{proof}

\subsection{Asymptotic form}\label{3.7}
Our last task is to show that a Ricci flat metric which is
projectively compact of order one admits an asymptotic form as
discussed in Section \ref{2.3}, at least around points in the open
curved orbits in the boundary. (In particular, this is always true if
the initial metric is Riemannian.) As we have noted in \ref{2.3}, the
asymptotic form will only be available for specific defining
functions, so we have to specialize the defining function
appropriately.

\begin{lemma}\label{lem3.7}
In the setting of Theorem \ref{thm3.6} assume that $U\cap\partial
M\subset\partial M_+\cup\partial M_-$. Then, possibly shrinking $U$,
one can modify the defining function $\rho$ to $\tilde\rho$ in such a
way that the boundary value of $\tilde\rho^{-3}g^{ij}\tilde\rho_j$
vanishes identically and the function
$\tilde\rho^{-4}g^{ij}\tilde\rho_i\tilde\rho_j$ is of the form
$\nu_0+\tilde\rho^2\nu_2$ for a non--zero constant $\nu_0$ and a smooth
function $\nu_2$ on $U$.
\end{lemma}
\begin{proof}
By assumption, the boundary value $\tau^{ij}$ of $\rho^{-2}g^{ij}$
satisfies $\tau^{ij}\rho_j=0$ and it has rank $n$. This implies that
viewing $\tau^{ij}$ as a map from the cotangent space to the tangent
space, its image will be the full annihilator of $\rho_j$. Since the
boundary value $\la^i$ of $\rho^{-3}g^{ij}\rho_j$ satisfies
$\la^i\rho_i=0$, we see that there is a one--form $\ph_j$ such that
$\tau^{ij}\ph_j=-\la^i$ and $\ph_j$ is actually unique up to adding a
some function times $\rho_j$. 

Now denoting by $\hat\nabla$ the connection in the projective class
corresponding to the defining function $\rho$, we can determine what
the fact that $L(\si^{-2}g^{-1})$ is parallel means in the splitting
determined by $\hat\nabla$. Using the formula for the tractor
connection on $S^2\Cal T$ \cite[formula (7)]{CGM} and the fact that
$\si^{-2}\rho^2$ is parallel for $\hat\nabla$, we get
\begin{align}
\hat\nabla_a(\rho^{-2}g^{ij})&=\rho^{-3}\delta^i_ag^{jk}\rho_k+\rho^{-3}\delta^j_ag^{ik}\rho_k.\label{tauder}\\
\hat\nabla_a(\rho^{-3}g^{ik}\rho_k)&=\rho^{-4}\delta^i_ag^{jk}\rho_j\rho_k-\rho^{-2}\hat\Rho_{aj}g^{ij}.\label{muder}\\
\hat\nabla_a\rho^{-4}g^{jk}\rho_j\rho_k&=-2\rho^{-3}\hat\Rho_{aj}g^{jk}\rho_k.\label{nuder}
\end{align}
Now we use this to compute
$\tau^{ai}\tau^{bj}\hat\nabla_i\ph_j=\tau^{ai}\hat\nabla_i\tau^{bj}\ph_j-
\tau^{ai}\ph_j\hat\nabla_i\tau^{bj}$. In the first summand, we just
get a linear combination of $\tau^{ab}$ and
$\tau^{ai}\hat\Rho_{ij}\tau^{jb}$. Since $\hat\nabla$ is a special
affine connection, $\hat\Rho_{ij}$ is symmetric, so both this terms
are symmetric in $a$ and $b$. From the second summand, we get
$-\tau^{ai}\ph_j\delta^b_i\la^j-\tau^{ai}\ph_j\delta^j_i\la^b$ so this
again adds a multiple of $\tau^{ab}$ plus $\la^a\la^b$. Thus we
conclude that $\tau^{ai}\tau^{bj}\hat\nabla_i\ph_j$ is symmetric in
$a$ and $b$, or otherwise put, the alternation of $\hat\nabla_i\ph_j$,
which equals $d\ph$ as $\hat\nabla$ is torsion free, contracts
trivially into $\tau^{ai}\tau^{bj}$. Since $\rho_j$ spans the null
space of $\tau^{ab}$, this implies that $d\ph=\ps\wedge d\rho$ for
some one--form $\ps$. In particular, this shows that the restriction
of $\ph$ to a one--form on $\partial M$ is closed. Possibly shrinking
$U$, we may assume that there is a smooth function $f:U\cap\partial
M\to\Bbb R$ such that $\ph|_{\partial M}=df$. Extending $f$
arbitrarily to $U$, we conclude that the one--form $f_i:=df$ has the
property that $\tau^{ij}f_j=-\la^i$.

Now we define $\tilde\rho:=e^f\rho$, which implies that
$\tilde\rho_j=\tilde\rho f_j+e^f\rho_j$. Thus we get 
$$
\tilde\rho^{-3}g^{ij}\tilde\rho_j=\tilde\rho^{-2}g^{ij}f_j+
e^f\tilde\rho^{-3}g^{ij}\rho_j=e^{-2f}(\rho^{-2}g^{ij}f_j+\rho^{-3}g^{ij}\rho_j), 
$$ and the term in the bracket goes to zero at the boundary by
construction. Next, consider
$\tilde\rho^{-4}g^{ij}\tilde\rho_i\tilde\rho_j=:\nu$. The analog of
\eqref{nuder} for $\tilde\rho$ implies that $\tilde\nabla_a\nu$
vanishes identically along the boundary. In particular, $\nu$ equals
some constant $\nu_0$ along the boundary, and this constant must be
non--zero, since we know that $L(\si^{-2}g^{-1})$ has rank $n+1$
everywhere. Hence $\nu-\nu_0$ vanishes along the boundary and thus is
 of the
form $\tilde\rho\nu_1$ for some function $\nu_1$ which is smooth up to
the boundary. Differentiating, we get $d\nu=\tilde\rho
d\nu_1+\nu_1d\tilde\rho$. But we know that $d\nu$ vanishes identically
along the boundary, so inserting a vector field $\xi$ such that
$d\tilde\rho(\xi)=1$, we see that $\nu_1$ vanishes along the boundary
and thus can be written as $\tilde\rho\nu_2$ for a function $\nu_2$
which is smooth up to the boundary.
\end{proof}

\begin{prop}\label{prop3.7}
In the setting of Theorem \ref{thm3.6} assume that the defining function
$\rho$ has the additional properties derived for $\tilde\rho$ in Lemma
\ref{lem3.7} above. Then putting $\nu=\rho^{-4}g^{ab}\rho_a\rho_b$, the
tensor field $h=\rho^2 g+\tfrac{1}{\nu}\tfrac{d\rho\odot
  d\rho}{\rho^2}$ satisfies the hypothesis of Theorem \ref{thm2.3} (for
$\al=1$).
\end{prop}
\begin{proof}
By assumption, $\rho^{-3}g^{ij}\rho_j$ goes to zero on the boundary,
so it is of the form $\rho t^i$ for some vector field $t^i$ which is
smooth up to the boundary. Moreover, by construction $t^i\rho_i=\nu$.
Now we define a tensor field $h^{ij}$ on $U$ by
\begin{equation}\label{hdef}
h^{ij}:=\tfrac{1}{\rho^2}g^{ij}-\tfrac{\rho^2}{\nu}t^it^j. 
\end{equation}
By Theorem \ref{thm3.6}, this is smooth up to the boundary and the
boundary value coincides with the one of $\rho^{-2}g^{ij}$ and thus
has rank $n$ with its null--space spanned by $\rho_j$. By definition,
$t^j\rho_j=\nu$, which shows that
$h^{ij}\rho_j=\tfrac1{\rho^2}g^{ij}\rho_j-\rho^2t^i=0$. On the other
hand, on the kernel of $t^i$, $h^{ij}$ evidently coincides with
$\tfrac1{\rho^2}g^{ij}$ so it is non--degenerate there.

Now from \eqref{hdef}, we see that
$g^{ij}=\rho^2h^{ij}+\tfrac{\rho^4}{\nu}t^it^j$. This represents an
orthogonal decomposition of $T^*M$ with respect to $g^{ij}$ into the
line spanned by $d\rho$ and the kernel of $t^i$, so the decomposition
of the space extends to the boundary. Dually, one obtains an
orthogonal decomposition for $g_{ij}$ into $\ker(d\rho)$ and the line
spanned by $t^i$. With respect to this decomposition, the metric
$g_{ij}$ then clearly  is the sum of $\rho^{-2}h_{ij}$ (where
$h_{ij}$ is the inverse of $h^{ij}$ on $\ker(d\rho)$, extended by zero
on the line spanned by $t^i$) and some multiple of
$\rho_i\rho_j$. This multiple can be computed by observing that
$g_{ij}t^it^i=\rho^{-4}\nu$, which shows that
$$
g_{ij}=\rho^{-2}h_{ij}+\rho^{-4}\nu^{-1}\rho_i\rho_j. 
$$ To complete the proof, it thus suffices to show that the function
$\nu^{-1}$ satisfies the assumptions of Theorem \ref{thm2.3}. But
$d\nu^{-1}=-\nu^{-2}d\nu$, so this vanishes identically along the
boundary. As in the proof of Lemma \ref{lem3.7}, we thus conclude that
$\nu^{-1}=(1/\nu_0)+\rho^2\tilde\nu$ for some function $\tilde\nu$
which is smooth up to the boundary. Thus $d\nu^{-1}=2\rho\tilde\nu
d\rho+\rho^2d\tilde\nu$. For a vector field $\ze$ on $U$ such that
$d\rho(\ze)=0$, we thus have $\ze\cdot\nu^{-1}=\rho^2(\ze\cdot\tilde\nu)$,
which completes the proof.
\end{proof}

\end{document}